\documentclass[13ptreqno]{amsart}
\usepackage{amssymb}
\usepackage{amsfonts} 
\usepackage{latexsym}   
\usepackage{amssymb}    
\usepackage{amsmath}    
\usepackage{amsbsy}
\usepackage{amsgen}
\usepackage{amsfonts}
\usepackage{array}
\usepackage{epsfig}
\usepackage{psfrag}
\usepackage[all]{xy}
\usepackage{hyperref}

\usepackage{amssymb}
\usepackage{mathabx}
\usepackage{amsfonts}
\usepackage{hyperref}
\usepackage{amsmath}
\usepackage{graphicx}
\usepackage[all]{xy}%
\providecommand{\U}[1]{\protect\rule{.1in}{.1in}}

\theoremstyle{plain}

\newtheorem{algorithm}{Algorithm}[section]
\newtheorem{thm}{Thm}

\newtheorem{corollary}[algorithm]{Corollary}

\newtheorem{definition}[algorithm]{Definition}
\newtheorem{example}[algorithm]{Example}

\newtheorem{lemma}[algorithm]{Lemma}

\newtheorem{theorem} [algorithm] {Theorem}
\newtheorem{theoremlet}[thm]{Theorem}
\newtheorem{theoremlet'}[thm]{Theorem$'$}

\newtheorem{corollarylet}[thm]{Corollary}

\newtheorem{proposition}[algorithm]{Proposition}
\newtheorem{remark}[algorithm]{Remark}

\numberwithin{equation}{algorithm}
\newtheorem{observe}[algorithm]{Observation}

\newtheorem*{thmc1}{Theorem C$'$}

\newtheorem{thmcst}[algorithm]{Cross Sectioning Theorem}
\newtheorem{ectthm}[algorithm]{Equivariant Classification Theorem}

\newtheorem*{msrconj}{Maximal Symmetry Rank Conjecture}

\newtheorem*{observe*}{Observation}

\def\qqq{\mathbb{Q}}
\def\rrr{\mathbb{R}}
\def\ccc{\mathbb{C}}
\def\zzz{\mathbb{Z}}

\DeclareMathOperator{\Fix}{Fix}
\def\codim{\textrm{codim}}
\def\bdm{\begin{displaymath}}
\def\edm{\end{displaymath}}
\def\beq{\begin{equation}}
\def\eeq{\end{equation}}
\def\bes{\begin{equation*}}
\def\ees{\end{equation*}}
\def\epcm{\end{picture}\end{center}\end{minipage}}
\def\bpcm{\begin{minipage}{80pt}\begin{center}\begin{picture}}

\def\t2{T^2}

\def\f4{F_4}
\def\g2{G_2}

\def\p2{\frac{\pi}{2}}

\def\Fix{\textrm{Fix}}
\def\rk{\textrm{rk}}

\def\dim{\textrm{dim}}

  \numberwithin{figure}{section}

\DeclareMathOperator{\curv}{curv}

\newtheorem*{ack}{Acknowledgements}

\begin{document}
\newcommand{\comment}[1]{\vspace{5 mm}\par \noindent
\marginpar{\textsc{Note}}
\framebox{\begin{minipage}[c]{0.95 \textwidth}
#1 \end{minipage}}\vspace{5 mm}\par}

\title[Torus Actions, Maximality, and Non-negative Curvature]{Torus Actions, Maximality, and Non-negative Curvature}

\author[Escher]{Christine Escher}
\address[Escher]{Department of Mathematics, Oregon State University, Corvallis, Oregon}
\email{tine@math.orst.edu}

\author[Searle]{Catherine Searle}
\address[Searle]{Department of Mathematics, Statistics, and Physics, Wichita State University, Wichita, Kansas}
\email{searle@math.wichita.edu}


\date{\today}

%

\begin{abstract} Let $\mathcal{M}_{0}^n$ be the class of closed, simply connected, non-negatively curved Riemannian manifolds admitting an isometric, effective, isotropy-maximal torus action. We prove that if $M\in \mathcal{M}_{0}^n$, then  $M$ is equivariantly diffeomorphic to the free, linear quotient by a torus of  a product of spheres of dimensions greater than or equal to three. As an immediate consequence, we prove the Maximal Symmetry Rank Conjecture for  all $M\in \mathcal{M}_{0}^n$. Finally, we show
the Maximal Symmetry Rank Conjecture for simply connected, non-negatively curved manifolds holds for dimensions less than or equal to nine without additional assumptions on the torus action.

\end{abstract}
\maketitle


\section{Introduction}

The classification of closed Riemannian manifolds with positive or non-negative sectional curvature is a long-standing 
problem in Riemannian geometry.  One successful approach has been the introduction of symmetries, and an important first case to understand is that of continuous abelian symmetries, that is, of torus actions.   
Our main result, stated below in Theorem \ref{t:thma}, gives an equivariant diffeomorphism classification
of closed, simply connected, non-negatively curved Riemannian $n$-manifolds with an isometric, effective and isotropy-maximal 
$T^k$-action, with $k\geq \lfloor\frac{n+1}{2}\rfloor$. 
This result significantly strengthens a recent classification up to equivariant rational homotopy equivalence of a subclass of these manifolds (cf. Theorems D and B in Galaz-Garc\'ia, Kerin, Radeschi, and Wiemeler \cite{GGKRW}). It also 
 gives us a full generalization of work
of Wiemeler \cite{Wie}  classifying  such manifolds  in the case of equality when the manifold is even-dimensional, that is, where 
$n=2m$  and  
$k=\lfloor\frac{n+1}{2}\rfloor=m$.

\begin{theoremlet}\label{t:thma}  Let $T^k$ act isometrically and effectively on $M^n$, a closed, simply connected, non-negatively curved Riemannian manifold. Assume that the action is isotropy-maximal.  Then $M$ is equivariantly diffeomorphic to a quotient of a free linear torus action of 
$$\mathcal{Z}^m=\prod_{i<r} S^{2n_i} \times \prod_{i\geq r} S^{2n_i-1},\, \, n_i\geq 2, \,\,   \textrm{where} \,\,\,
 n\leq m\leq 3n-3k\,.$$
\end{theoremlet}

One defines a $T^k$-action on a smooth manifold, $M^n$, to be {\it isotropy-maximal},  when there exists a point in $M$ whose isotropy group is maximal, namely of dimension $n-k$. Indeed, a $T^k$-action on $M$ is isotropy-maximal if and only if there exists a minimal orbit, that is an orbit of dimension $2k-n$, which implies  
$k\geq \lfloor\frac{n+1}{2}\rfloor$. 

By work of Grove and Searle \cite{GS}, for closed, simply connected manifolds of strictly positive sectional curvature,  an isotropy-maximal action only occurs when $k=\lfloor\frac{n+1}{2}\rfloor$, 
and  such manifolds are equivariantly diffeomorphic to 
spheres or complex projective spaces with linear torus actions. 
In fact, this case corresponds to the maximal symmetry rank case, where the {\it symmetry rank} of a manifold is defined to be the rank of the isometry group of $M$.

In contrast, the corresponding maximal symmetry rank for manifolds of non-negative curvature was conjectured  to be 
approximately two-thirds the dimension of the manifold by Galaz-Garc\'ia and Searle \cite{GGS1, GGS2}.  In \cite{GGS1}, they obtained a diffeomorphism classification, but only in dimensions less than or equal to $6$.
Based on our results here, we reformulate and sharpen the conjecture (cf. \cite{GGS1}).

\begin{msrconj}\label{msrconj} Let $T^k$ act isometrically and effectively on
$M^n$, a closed, simply connected, non-negatively curved Riemannian manifold. Then  the following hold:
\begin{enumerate}
\item 
$k\leq \lfloor 2n/3\rfloor$; and 

\item When $k= \lfloor 2n/3\rfloor$, $M^n$ is equivariantly diffeomorphic to  
$\mathcal{Z}/T^m$,  where 
$$\mathcal{Z}=  \prod_{i\leq r} S^{2n_i-1} \times\prod_{i>r} S^{2n_i},$$
 with  $n_i\geq 2, \,\, r= 2\lfloor 2n/3\rfloor-n, \,\,0 \leq m \leq 2n \mod 3,$
and the $T^m$ action is  free and linear. \end{enumerate}
\end{msrconj}
We will show that the upper bound on the rank of the action in Part (1) of the Maximal Symmetry Rank Conjecture is exactly the upper bound on the rank of an isotropy-maximal action on a closed, simply connected, non-negatively curved Riemannian manifold. Combining this result with the lower bound for an isotropy-maximal action we obtain the following corollary of Theorem \ref{t:thma}.

\begin{corollarylet}\label{corb} Let $T^k$ act isometrically and effectively on $M^n$, a closed, simply connected, non-negatively curved Riemannian manifold. Assume that the action is 
isotropy-maximal. Then  $  \lfloor(n+1)/2\rfloor 
\leq k\leq \lfloor 2n/3\rfloor$.
\end{corollarylet} 

Thus,  there are two extremal cases for Theorem \ref{t:thma}: the case of torus manifolds,  where the rank of the action is  half the dimension of the manifold, and the case of those with maximal symmetry rank, which corresponds to the case when the rank  is approximately two-thirds the dimension of the manifold. 

\begin{observe*}
The so-called LVMB manifolds  are 
examples of closed, simply connected, complex manifolds,  of complex dimension $n-m$, with $2m\leq n$, admitting rank $n$ isotropy-maximal torus actions, with $n> \lfloor 2(2n-2m)/3\rfloor$ (see Ishida \cite{I}). Corollary \ref{corb} tells us that these manifolds cannot admit an invariant metric of non-negative sectional curvature.  

\end{observe*}

As a direct application of Theorem \ref{t:thma} in combination with Lemma 2.2 from \cite{I}  (Lemma \ref{ishida} in this article),  we obtain the Maximal Symmetry Rank conjecture in the presence of an isotropy-maximal 
action.

\begin{theoremlet}\label{thmc}
The Maximal Symmetry Rank Conjecture holds for an isotropy-maximal action.
\end{theoremlet}

In fact, the proof of  Theorem \ref{thmc} tells us that we may reformulate the Maximal Symmetry Rank Conjecture as follows.

\begin{msrconj} Let $T^k$ act isometrically and effectively on
$M^n$, a closed, simply connected, non-negatively curved Riemannian manifold.
Then the action is isotropy-maximal for $k=\lfloor 2n/3\rfloor$.
\end{msrconj}

Our results also have implications for the class of rationally elliptic manifolds. In Proposition \ref{p:free rank}, we find a lower bound for the free rank of a $T^k$-action, where the {\em free rank} is the dimension of the largest subtorus that can act almost freely. Combining this result with the upper bound for the free rank obtained in  Galaz-Garc\'ia, Kerin and Radeschi \cite{GGKR}, allows us to reformulate  Theorem \ref{thmc} as follows:

\begin{thmc1}  The Maximal Symmetry Rank Conjecture holds for rationally elliptic manifolds.
\end{thmc1}

Observe that if the Bott Conjecture were true, then we would no longer need to assume rational ellipticity in Theorem C$'$ nor that the action is isotropy-maximal in  Corollary \ref{corb}.  That is,  if the Bott Conjecture were true, then the Maximal Symmetry Rank Conjecture would hold for all simply connected, closed, non-negatively curved $n$-manifolds admitting an isometric $T^k$-action.

As mentioned above, Theorem \ref{t:thma} gives us important information for closed, simply connected, non-negatively curved manifolds of maximal symmetry rank.  Recall first that 
such manifolds have been classified up to diffeomorphism in dimensions less than or equal to $6$ in \cite{GGS1} and up to equivariant diffeomorphism in dimensions $4$ through $6$ in Galaz-Garc\'ia and Kerin \cite{GGK}. 
 In Proposition \ref{max} we show that a cohomogeneity three torus action must be isotropy-maximal for this class of manifolds, provided the dimension is greater than or equal to $7$. 
 We can then apply Theorem \ref{t:thma} to extend this equivariant diffeomorphism classification to higher dimensions.

\begin{theoremlet}\label{MSRNN}  The Maximal Symmetry Rank Conjecture holds for $n=7, 8, 9$.
\end{theoremlet}

As an immediate consequence of Corollary \ref{corb} combined with Proposition \ref{max}, we confirm Part (1)  of the Maximal Symmetry Rank Conjecture for dimensions less than or equal to $12$.

   \begin{corollarylet}   Part (1) of the Maximal Symmetry Rank Conjecture holds  for $n\leq 12$.   \end{corollarylet}

It is also of interest to  classify closed, simply connected, non-negatively curved Riemannian  $n$-dimensional manifolds of almost maximal symmetry rank, that is, admitting a $T^k$ isometric, effective action of rank $k=\lfloor 2n/3\rfloor -1$. 
In a separate article \cite{ES1}, the authors use some of the tools developed to prove Theorem \ref{t:thma}, in combination with results about cohomogeneity three torus actions, to obtain 
 a classification of $6$-dimensional, closed, simply connected, non-negatively curved manifolds of almost maximal symmetry rank, thereby extending the almost maximal symmetry rank classification work of Kleiner \cite{K} and Searle and Yang \cite{SY} in dimension $4$ and work of Galaz-Garc\'ia and Searle \cite{GGS2} in dimension $5$. We expect that  many of the results used in the proof of Theorem \ref{t:thma}, as well as Theorem \ref{t:thma} itself,  should admit a number of similar applications and should be particularly useful for any classification of closed, simply connected, non-negatively curved manifolds of higher dimension of maximal or almost maximal symmetry rank.

Finally, Theorem \ref{t:thma}, Corollary \ref{corb}, and Theorem \ref{thmc} can all be extended to include $T^k$-actions, for   $k=\lfloor 2n/3\rfloor$, that are {\em almost isotropy-maximal}, that is, for which there exists a point in $M$ whose isotropy group is {\em almost} maximal, namely of dimension $n-k+1$. Moreover,  a forthcoming paper of Dong, Escher, and Searle \cite{DES},  
generalizes Theorem \ref{t:thma} to
almost isotropy-maximal actions when $k<\lfloor 2n/3\rfloor$.

\subsection{Organization}

We have organized the paper in general so as to present the topological tools and results  first, followed by their geometric counterparts. In Section \ref{s2}, we describe the topological and geometric tools we will need to prove Theorem \ref{t:thma}, Corollary \ref{corb} and Theorem \ref{MSRNN}. In Section \ref{s3}, we prove a generalization of the Equivariant Cross-Sectioning Theorem of Orlik and Raymond and thereby obtain an Equivariant Classification Theorem. 
 In Section \ref{s4}, we generalize results on torus manifolds of non-negative curvature to the class of  torus manifolds with non-negatively curved quotient spaces. This class comprehends almost non-negatively curved torus manifolds with non-negatively curved quotient spaces.  In Section \ref{generallowerbound}, we find a general lower bound for the free rank of an action on an Alexandrov space with a lower curvature bound.  In Section \ref{s:6}, we prove Theorem \ref{t:thma}.  In Section \ref{s5}, we prove that an almost isotropy-maximal action of rank  $\lfloor 2n/3\rfloor$ is isotropy-maximal, thus proving that 
  Theorem \ref{t:thma}, Corollary \ref{corb}, and Theorem \ref{thmc}  all  extend to include $T^k$-actions, for   $k=\lfloor 2n/3\rfloor$, that are almost isotropy-maximal.   Finally,  in Section \ref{s7}, we prove  Corollary \ref{corb}, 
Theorem \ref{MSRNN}, and we present a significantly streamlined proof of the Maximal Symmetry Rank Conjecture for dimensions less than or equal to $6$. 
  
\begin{ack} We  thank  Karsten Grove, Michael Wiemeler, Fred Wilhelm, Wolfgang Ziller and the referees for many helpful comments and suggestions for improvement.  We are grateful to Mark Walsh for his generous help with the figures.
C. Escher acknowledges partial support  from  the Oregon State University College of Science Scholar Fund and from the Simons Foundation (\#585481, C. Escher). C. Searle  acknowledges partial support from a Wichita State University ARCs grant \#150353 and support  from grants from the National Science Foundation (\#DMS-1611780 and \#DMS-1906404), as well as  from the Simons Foundation (\#355508, C. Searle). This material is based in part upon work supported by the National Science Foundation under Grant No. DMS-1440140 while the authors were in residence at the Mathematical Sciences Research Institute in Berkeley, California, during the Spring 2016 semester.
\end{ack}


\section{Preliminaries}\label{s2}

In this section we will gather basic results and facts about transformation groups, torus actions, orbit spaces, torus manifolds and orbifolds, Alexandrov geometry, as well as results concerning $G$-invariant manifolds of non-negative and almost non-negative sectional curvature.


\subsection{Transformation Groups}\label{2.1}

Let $G$ be a compact Lie group acting on a smooth manifold $M$. We denote by $G_x=\{\, g\in G : gx=x\, \}$ the \emph{isotropy group} at $x\in M$ and by $G(x)=\{\, gx : g\in G\, \}\simeq G/G_x$ the \emph{orbit} of $x$.  Orbits will be {\em principal}, {\em exceptional}, or {\em singular}, depending on the relative size of their isotropy subgroups; namely, principal orbits correspond to those orbits with the smallest possible isotropy subgroup, an orbit is called exceptional when its isotropy subgroup is a finite extension of the principal isotropy subgroup, and singular when its isotropy subgroup is of strictly larger dimension than that of the principal isotropy subgroup.

The \emph{ineffective kernel} of the action is the subgroup $K=\bigcap_{x\in M}G_x$. We say that $G$ acts \emph{effectively} on $M$ if $K$ is trivial. The action is called \emph{almost effective} if $K$ is finite. The action is \emph{free} if every isotropy group is trivial and \emph{almost free} if every isotropy group is finite. 
As mentioned in the Introduction, the {\em free rank} of an action is the rank of the maximal subtorus that acts almost freely. In order to further distinguish between the case when the free rank corresponds to a free action and the case when it corresponds to an almost free action, we make the following definition.

\begin{definition}[{\bf Free Dimension}] Suppose that the free rank of a $T^k$-action is equal to $r$ and let $T^r$ denote a maximal subtorus of $T^k$ acting almost freely.   
We define the {\em free dimension} to be the dimension of the largest subtorus of any such $T^r$ that acts freely.
\end{definition}

There is a natural stratification of $M$ into smooth submanifolds induced by the group action. The
stratum, $S,$ of $x\in M$ is defined to be%
\[
S\equiv\left\{  y\in M\text{ }|\text{ }\exists g\in G\text{ with }G_{x}%
=gG_{y}g^{-1}\right\}  .
\]
By definition, $S$ is $G$-invariant and $G$ acts principally on $S$. 
This stratification leads to a natural stratification on the quotient space $M/G$. Let $\pi:M\rightarrow M/G$,  and let $S$ be a stratum of $x\in M$, then $\pi(S)=S^*$ is an {\em orbit stratum} of $\pi(x)$ in $M/G$. 

 We will sometimes denote the \emph{fixed point set} $M^G=\{\, x\in M : gx=x, g\in G \, \}$ of the $G$-action by $\Fix(M ; G )$. We define its dimension as $$\dim(\Fix(M;G))=\max \{\,\dim(N): \text{$N$ is a connected component of $Fix(M;G)$}\,\}.$$ One measurement for the size of a transformation group $G\times M\rightarrow M$ is the dimension of its orbit space $M/G$, also called the {\it cohomogeneity} of the action. This dimension is clearly constrained by the dimension of the fixed point set $M^G$  of $G$ in $M$. In fact, $\dim (M/G)\geq \dim(M^G) +1$ for any non-trivial, non-transitive action. In light of this, the {\it fixed point cohomogeneity} of an action, denoted by $\textrm{cohomfix}(M;G)$, is defined by
\[
\textrm{cohomfix}(M; G) := \dim(M/G) - \dim(M^G) -1\geq 0.
\]
A manifold with fixed point cohomogeneity $0$ is also called a {\it $G$-fixed point homogeneous manifold}. For product groups, we define the following refinement of 
a fixed point homogeneous action.

\begin{definition}[{\bf Nested Fixed Point Homogeneous}]\label{NFPH}
Let $G=H_1\times \cdots \times H_l=H^l$ 
act isometrically and effectively on $M^n$. Without loss of generality, we suppose that the $H_i$  are ordered so that the following holds.
The subgroup $H_1$ acts fixed point homogeneously on $M$ and we denote by $N_1$  the connected component of largest dimension in $M^{H_1}$.
The subgroup $H_2$ acts fixed point homogeneously on $N_1$ and  we denote by $N_2$  the connected component of largest dimension in $N_1^{H_2}$.
Continuing in this fashion, we assume that the subgroup $H_i$ acts fixed point homogeneously on $N_{i-1}$ and denote by $N_{i}$  the connected component of largest dimension in $N_{i-1}^{H_i}$ for each $1\leq i\leq l$.

We will 
 call a manifold 
{\em nested $H$-fixed point homogeneous} when these conditions hold, and note that this implies the existence of a tower of nested $H$-fixed point sets, 
$$N_l\subset N_{l-1}\subset \cdots \subset N_1\subset N_0 = M.$$ 

\end{definition}
\begin{observe}
Since each $N_i\subset N_{i-1}$ is of codimension $d$, the number $l$ of $H$ factors is restricted by the dimension $n$ of the manifold, namely, $ld\leq n$. 
\end{observe} 

\begin{example} The sphere with an isometric torus action of maximal symmetry rank is nested $S^1$-fixed point homogeneous.
\end{example}


\subsection{Torus Actions} \label{s2.1}
 In this subsection we will recall notation and facts about smooth $G$-actions on smooth $n$-manifolds, $M$, in the special case when $G$ is a torus.
 We first recall the definition of an {\em isotropy-maximal torus action}, introduced in  \cite{I} as {\em maximal} and renamed in \cite{GGKRW} as {\em slice-maximal}.    

We have chosen  instead to 
call these actions {\em isotropy-maximal}, since the rank of the largest possible isotropy subgroup  for a cohomogeneity $n-k$ torus action is $n-k$, corresponding to an isotropy-maximal action.

\begin{definition}[{\bf Isotropy-Maximal Action}]  Let $M^n$ be a connected manifold with an effective $T^k$-action.  
\begin{enumerate}  
\item We call the $T^k$-action on $M^n$ {\em isotropy-maximal} if there is a point $x \in M$ such that 
the dimension of its isotropy group is $n-k$, that is,  $\dim(T^k_x) = n-k$.
\item  The orbit $T^k(x)$ through $x \in M$ is called {\em minimal} if $\dim(T^k(x)) = 2k-n$.
\end{enumerate} 
\end{definition}

Note that the action of $T^k$ on $M$ is isotropy-maximal if and only if there exists a minimal orbit $T^k(x)$.  The following lemma of \cite{I} shows that an isotropy-maximal 
action on $M$ implies that there is no larger torus acting on $M$ effectively.

\begin{lemma}\label{ishida} \cite{I}  Let $M$ be a connected manifold with an effective $T^k$-action.  Let $T^l \subset T^k$ be a subtorus of $T^k$.  Suppose 
that the action of $T^k$ restricted to $T^l$ on $M$ is isotropy-maximal.  Then $T^l = T^k$.  
\end{lemma}

We also define the concept of an almost isotropy-maximal action.
\begin{definition}[{\bf Almost Isotropy-Maximal Action}]  Let $M^n$ be a connected manifold with an effective $T^k$-action.  
\begin{enumerate}  
\item We call the $T^k$-action on $M^n$ {\em almost isotropy-maximal} if there is a point $x \in M$ such that 
the dimension of its isotropy group is $n-k-1$, that is,  $\dim(T^k_x) = n-k-1$.
\item  The orbit $T^k(x)$ through $x \in M$ is called {\em almost minimal} if $\dim(T^k(x)) = 2k-n+1$.
\end{enumerate} 
\end{definition}

Note that the action of $T^k$ on $M$ is almost isotropy-maximal if and only if there exists an almost minimal orbit $T^k(x)$.  
Lemma \ref{ishida} generalizes to almost isotropy-maximal actions.

\begin{lemma} Let $M$ be a connected manifold with an effective $T^k$-action.  Let $T^l \subset T^k$ be a subtorus of $T^k$.  Suppose 
that the action of $T^k$ restricted to $T^l$ on $M$ is almost isotropy-maximal.  Then $l=k$ or $l=k-1$.
\end{lemma}
In \cite{I},   the following properties of an isotropy-maximal action are obtained.

 \begin{lemma} \cite{I}\label{props}   Let $M$ be a connected manifold with an isotropy-maximal $T^k$ action.  Let $T^k(x)$ be a minimal orbit and let $T^k_x$ be the corresponding isotropy subgroup of $x\in M$.  Then 
 \begin{enumerate}
 \item The isotropy group $T^k_x$ at $x$ is connected.
 \item $T^k(x)$ is a connected component of the fixed point set of $T^k_x$.
 \item  Each minimal orbit is isolated.  
  \end{enumerate}
  In particular, there are finitely many minimal orbits if $M$ is compact.
 \end{lemma}
 
 \begin{observe} Note that Parts (2) and (3) of Lemma \ref{props} are equivalent. Moreover, via the proof of Lemma \ref{props} in \cite{I}, or more simply, via the Slice theorem and the Maximal Symmetry Rank Theorem of Grove and Searle \cite{GS} applied to the normal sphere of an almost  minimal orbit, it is easy to see that Property (1)  also holds for an almost isotropy-maximal action. Note however that for almost isotropy-maximal actions, neither Part (2) nor (3) hold in general. Consider an almost isotropy-maximal $T^2$ action on $S^1\times S^3$ with the first circle acting freely on $S^1$ and the second circle fixing a circle in the $S^3$. Then the almost minimal orbits, $T^2/T^1$, are strictly contained in  $\Fix(S^1\times S^3; T^1)=T^2$, and thus are not isolated.
 \end{observe}

An important class  of $T^k$-actions on $n$-dimensional manifolds is the class of  {\em locally standard} torus actions, which we now define.

 \begin{definition}[{\bf Locally Standard}]\label{ls} 
A $T^k$ action on $M^n$ is called {\em locally standard} if for 
 each point $x \in M$, there is a neighborhood which is $T^k$-equivariantly diffeomorphic to
$$T^r \times W \times \rrr^m,$$
where $r=k-\dim(T^k_x)$, $W$ is a faithful $T^k_x$-representation of real dimension $2\dim(T^k_x)$,   and $T^k\cong T^r\times T^k_x$ acts trivially on $\rrr^m$, $T^r$ acts trivially on $W$, and $T^k_x$ acts trivially on $T^r$.
\end{definition}

Finally, we recall the definitions of polar and infinitesimally polar group actions (see, for example,  \cite{GZ} and \cite{LT}, respectively).

\begin{definition} Let $G$ be a compact Lie group acting isometrically on an $n$-dimensional manifold, $M$. Then the action is said to be {\em polar} if there exists an immersed submanifold, $\Sigma$, that is, a so-called section, that meets all orbits orthogonally. A  $G$-action on $M$ is said to be {\em infinitesimally polar} if the slice representation at every point is polar.
\end{definition}


\subsection{Orbit Spaces}

To any orbit space, $M/G$, we may assign isotropy information, in the form of weights.
We recall the definition of a weighted orbit space for a smooth $G$-action on $M$.  

\begin{definition}[{\bf Weighted Orbit Space}] Let $G$ act smoothly on an $n$-manifold $M$ with orbit space $M^*=M/G$.  To each orbit in $M^*$ there is associated  
a certain orbit type which is characterized by the isotropy group of the points of the orbit together with the slice representation at the given orbit.   This orbit space together with its orbit types and slice representation is called a {\em weighted  orbit space}.  
\end{definition}

Let $G=T^k$ act isotropy-maximally on $M^n$. Since the action is isotropy-maximal, $\partial(M^n/G)$ is a union of codimension one subspaces that correspond to the images of codimension two fixed point set components of circle subgroups of $T^k$.  We note that it is enough to specify the weights of these codimension one subspaces, which we will call  {\em  facets}, 
 and together with a description of the orbit space, one then obtains a complete description of all orbit types.

Let
\begin{eqnarray*}
p&: & \rrr^k \longrightarrow  T^k \\
&&(x_1,\dots,x_k) \mapsto (e^{2\pi x_1 i},\dots, e^{2\pi x_k i})\\
\end{eqnarray*}
 be the universal covering projection 
and let $H$ be a circle subgroup of $T^k$.  Since each component of $p^{-1}(H)$ is a line containing at least two 
integer lattice points of $\rrr^k$, it is natural to parametrize $H$ as follows.  Let $a_1, \dots, a_k$ be relatively prime integers and 
let ${\mathbf a}=(a_1, \dots, a_k) \in \zzz^k$. We call the $a_i\in \zzz$ the weights of the corresponding circle isotropy subgroup $H(\mathbf a)$.
  Then $H(\mathbf a) = H(a_1,\dots,a_k)= \{(x_1,\dots,x_k)|x_i=a_i t \mod \zzz, \,0 \le t <1, \, i=1,\dots k\}$.  With this notation $H(\mathbf a)$
  is the image of a line in $\rrr^k$ through the origin and the lattice point $\mathbf a = (a_1,\dots,a_k)$ under the projection $p$.  
  We define the matrix of $m$ isotropy groups $H(\mathbf a_1),\dots, H(\mathbf a_m)$ to be the following $m \times k$-matrix:
  \[
M_{(\mathbf a_1,\dots, \mathbf a_m)} = 
 \begin{bmatrix}
    a_{11} & a_{12} & a_{13} & \dots  & a_{1k} \\
    a_{21} & a_{22} & a_{23} & \dots  & a_{2k} \\
    \vdots & \vdots & \vdots & \ddots & \vdots \\
    a_{m1} & a_{m2} & a_{m3} & \dots  & a_{mk}
\end{bmatrix}.
\]
We will denote the weights of a torus action via its matrix $M_{(\mathbf a_1,\dots, \mathbf a_m)}$.

We will see that the quotient $M/T$ of an isotropy-maximal torus action is 
an 
{\it $n$-manifold with corners}, that is, a Hausdorff space together with a maximal atlas of local charts onto open subsets of the simplicial cone, $[0, \infty)^n\subset \rrr^n$, so that the overlap maps are homeomorphisms which preserve codimension (see, for example, \cite{D}).

We now recall the definition of a weight preserving diffeomorphism from  \cite{OR2}, adapted to the language of manifolds with corners, noting that a {\em weakly smooth} map between manifolds with corners is as defined in Joyce \cite{J}. Observe that a  weakly smooth map is smooth when restricted to the interior of a manifold with corners.

\begin{definition}[\bf Weight Preserving Diffeomorphism]\label{wpd} Let $M^*_1$ and $M^*_2$ be manifolds with corners obtained as the $G$-quotients of $M_1$ and $M_2$, respectively. We say that a weakly smooth, strata-preserving map with weakly smooth inverse from $M^*_1$ to $M^*_2$ which carries the weights of $M^*_1$  isomorphically onto the weights of $M^*_2$ is a {\em weight preserving diffeomorphism}.
\end{definition}

\begin{remark}
For an isotropy-maximal $T^k$-action on $M^n_i$, let $A_i$, $i=1, 2$ be the set of weights for each $M^*_i$ and assume that the $M^*_i$ are weight-preserving diffeomorphic. Then the weight-preserving diffeomorphism is given by
$$\phi: (M_1^*, A_1)\mapsto (M_2^*, A_2),$$
where $A_2=V\cdot A_1\cdot U$, with $V\in GL(k; \zzz)$ and $U\in GL(m; \zzz)$, $m$ is the number of strata of $M_1^*$ with $T^1$ isotropy, and  $\phi|_{M_1^*}$ is a  weakly smooth map with weakly smooth inverse.
\end{remark}

In the following proposition and corollary, we see that for a restricted class of torus actions 
 the isotropy subgroups   will span the torus. They are straightforward generalizations of Theorem 1.6  and Corollary 1.7 in  Kim, McGavran and Pak, \cite{KMP}, respectively and we leave the proofs to the reader.

\begin{proposition}\label{nonfree}
Let $T^k$ act on $M^n$ effectively, where $M^n$ is a closed, simply connected manifold of dimension $n$.  Suppose further that all isotropy subgroups are connected, all singular orbits correspond to points on the boundary of the quotient space,  and  $M^n/T^k=D^{n-k}$. Then no subgroup $T^l$, $k> l\geq 1$, can contain all elements of $T^k$ that act nonfreely on $M^n$.
\end{proposition}

\begin{corollary}\label{c:KMP} With the same hypotheses as in Proposition \ref{nonfree},  the isotropy subgroups of the $T^k$-action on $M^n$ span  $T^k$ and there are at least $k$ different circle isotropy subgroups of $T^k$. 
\end{corollary}

 The condition for $k$ circle subgroups
to be {\em generators of $T^k$} is contained in the following lemma from \cite{O}. 

\begin{lemma}\cite{O} \label{oh2}
  The $k$ circle subgroups generate $T^k$, that is, $H(\mathbf a_1) \times \cdots \times H(\mathbf a_k) \cong T^k$ if and only if 
 $\det(M_{(\mathbf a_1,\dots, \mathbf a_k)}) = \pm 1$.   

\end{lemma}

We  recall the definition of a conical orbit structure of a $G$-action on a space $X$.  

\begin{definition}[{\bf Conical Orbit Structure}] \label{defncone}
 Denote by  $K^{\circ}(Y)= Y \times [0,1) / (Y \times \{0\})$
the open cone over a space $Y$ and its closure by $K(Y)$.
 The orbit structure of  $X$ is called {\em conical},  if $X^*$ is homeomorphic to an open cone $K^{\circ}(Y)$ with constant orbit type along rays less 
           the vertex, $p^*$, and we say that $X^*$ is a {\em conical section}.  
 \end{definition}

For the case of a $T^k$-action on a closed manifold $M^n$,  the following direct consequence of the Slice Theorem was obtained in McGavran \cite{Mc} (see also \cite{Mc2}), classifying   a neighborhood of point in $M$ with isotropy group $T^l$, $l \le k$. 
\begin{theorem}\cite{Mc, Mc2} \label{cone}  Suppose $T^k$ acts locally smoothly on a closed manifold $M^n$.  Suppose $p \in M^n$ has isotropy group 
$T^l$,  $0\le l \le k$.
Let $X$ be a closed invariant neighborhood of $p$ in $M$ such that $X^* = K(Y)$.  Suppose the inverse image in $X$ of $K^{\circ}(Y)$ has conical orbit structure. 
Then 
$X$ is equivariantly homeomorphic to $T^{k-l} \times D^{n-k+l}$.  
\end{theorem}

\begin{remark}\label{conethm} If we assume the $T^k$ action is locally standard in Theorem \ref{cone}, then the action of the isotropy subgroup, $T_p=T^l$, on its normal slice, $D^{n-k+l}$, is polar, and so admits a section, that is, the $T^k$-action is infinitesimally polar.
\end{remark}


\subsection{Torus Manifolds} An important subclass of manifolds admitting an effective torus action are the so-called {\it torus manifolds}. 
For more details on torus manifolds, we refer the reader to Hattori and Masuda \cite{HM}, Masuda and Panov \cite{MP}, and  Buchstaber and Panov \cite{BP1}. 

Torus manifolds arose as a generalization of the concept of a {\it toric variety}, which is a normal algebraic variety, $M$, containing the algebraic torus $(\mathbb{C}^*)^n$ as a Zariski open subset in such a way that the natural action of ($\mathbb{C}^*)^n$ on itself extends to an action on $M$ (see Buchstaber and Panov \cite{BP} for more details). In particular, 
in \cite{DJ}, a topological counterpart to non-singular projective toric varieties was introduced, now called {\it quasitoric manifolds}, see Definition \ref{qt} below.  Originally they were named ``toric manifolds" but then were renamed in \cite{BP} since the term toric manifold is reserved in algebraic geometry for a ``non-singular toric variety".

 \begin{definition}[{\bf Torus Manifold}] A {\em torus manifold} $M$ is a $2n$-dimensional closed, connected, orientable, smooth manifold with an effective smooth action of 
an $n$-dimensional torus $T$ such that $M^T\neq \emptyset$.
\end{definition}

The $T^n$-action on $M^{2n}$ is an isotropy-maximal action. Further, $M^{2n}$ is  a $S^1$-fixed point homogeneous manifold and moreover, the action on $T^n$ on $M^{2n}$ is a nested $S^1$-fixed point homogeneous action, that is, we can find a tower of nested fixed point sets for each $p\in M^T$:
$$ \{p\}\subset F^2\subset \cdots \subset F^{2n-2}\subset M^{2n}.$$

However, not all torus manifolds are locally standard,  see  
Section 11 of Fukukawa, Ishida, and Masuda \cite{FIM} for examples.  The orbit space of a locally standard action is a compact connected $n$-manifold with corners with the property that 
every codimension $k$-face belongs to exactly $k$ facets.   Adding the assumption that the orbit space is acyclic and has acyclic faces 
 imposes strong topological restrictions, as we see in the following theorem of \cite{MP}. 
\begin{theorem}\label{MP}\cite{MP}
Let $M$ be a torus manifold. Then the odd degree integer cohomology of $M$ vanishes if and only if $M$ is locally standard and the orbit space $M/T$ is acyclic with acyclic faces.  
\end{theorem}

The quotient space of a $T^n$-manifold plays an important role in the theory. Recall that an $n$-dimensional  convex polytope is called {\it simple}  if the number of facets meeting at each vertex is $n$.  A {\em homology polytope}  is an $n$-manifold with corners that is acyclic with acyclic faces (see \cite{MP}). 
 A {\it nice} manifold with corners, or {\it a manifold with faces}, has every codimension $k$ face  contained in exactly $k$ facets.  Clearly, a simple convex polytope is a nice homology polytope.

The orbit space of a locally standard action of $T^n$ on $M^{2n}$ is an $n$-dimensional manifold with corners.  Quasitoric manifolds have the property that their orbit space is diffeomorphic, as a manifold with corners, to a simple polytope $P^n$.  Note that two simple polytopes are diffeomorphic as manifolds with corners if and only if they are combinatorially equivalent by work of  Davis \cite{D}, and Wiemeler \cite{Wie2}.   

\begin{definition} [{\bf Quasitoric manifold}]\label{qt} Given a combinatorial simple polytope $P^n$, a $T^n$ manifold $M^{2n}$ is called a {\em quasitoric manifold over $P^n$} if the following two conditions are satisfied:  
\begin{enumerate}
\item the $T^n$-action is locally standard; and  
\item there is a projection map $\pi: M^{2n} \longrightarrow P^n$ which is constant on $T^n$-orbits and which maps every $k$-dimensional orbit to a point in the interior of a codimension $k$ face of $P^n$ for $k=0,\dots,n$.
\end{enumerate}
\end{definition} 

This definition implies that the $T^n$-action on such a quasitoric manifold $M^{2n}$ is free over the interior of the orbit polytope $P^n$,  and the vertices of $P^n$ correspond to $T^n$-fixed points.  In particular,  
quasitoric manifolds are examples of torus manifolds.  Even though the conditions on the torus action in the case of a torus manifold are much weaker than the conditions in the case of  a quasitoric manifold, torus manifolds still admit a combinatorial treatment similar to quasitoric manifolds.  
For example, the orbit space of a torus manifold is a nice manifold with corners if the action is locally standard.   If, in addition, the orbit space is acyclic with acyclic faces, the constructions below for simple polytope orbit spaces can be generalized to this case, that is, to the case of orbit spaces that are nice homology polytopes.

Let $\pi: M^{2n} \longrightarrow P^n=M^{2n}/T^n$ be the orbit map of a quasitoric manifold and let $\mathcal{F} = \{F_1,...,F_m\}$ be the set of facets of $P^n$. Denote the preimages by $M_j =\pi^{-1}(F_j)$, $1\leq j \leq m$.
Points in the relative interior of a facet $F_j$ correspond to orbits with the same one-dimensional isotropy subgroup, which we denote by $T_{F_j}$ . Hence $M_j$ is a connected component of the fixed point set of the circle subgroup $T_{F_j} \subset T^n$. This implies that $M_j$ is a $T^n$-invariant submanifold of codimension $2$ in $M$, and $M_j$ is a torus manifold over $F_j$ with the action of the quotient torus $T^n/T_{F_j} \cong T^{n-1}$. Following the terminology of  Davis and Januszkiewisz \cite{DJ}, we refer to $M_j$ as the {\em characteristic submanifold} corresponding to the $j$th facet $F_j \subset P^n$.  The mapping
$\lambda:F_j  \rightarrow T_{F_j}, 1\leq j \leq m,$
is called the {\it characteristic function} of the torus manifold $M^{2n}$.
Now let $G$ be a codimension-$k$ face of $P^n$ and write it as an intersection of $k$ facets: $G=F_{j_1} \cap\cdots\cap F_{j_k}$. Assign to each face $G$ the subtorus $T_G = \prod_{F_i \supset G}  T_{F_i} \subset T^{\mathcal{F}}$.
Then $M_G =\pi^{-1}(G)$ is a $T^n$-invariant submanifold of codimension $2k$ in $M$, and $M_G$ is fixed under each circle subgroup $\lambda(F_{j_l} )$, $1\leq   l\leq   k$.

To each $n$-dimensional simple convex polytope, $P^n$, we may associate a $T^m$-manifold  $\mathcal{Z}_P$ with the orbit space $P^n$, as in  \cite{DJ}.

 \begin{definition}[{\bf Moment Angle Manifold}]\label{moment}
For every point $q \in P^n$, denote by $G(q)$ the unique (smallest)  face containing $q$ in its interior.
For  any simple polytope $P^n$ define the {\em moment angle manifold} $$\mathcal{Z}_P = (T^{\mathcal{F}} \times P^n)/\sim \,\, = \,\,(T^{m} \times P^n)/\sim \,,$$
where $(t_1,p) \sim (t_2, q)$ if and only if $p=q$ and $t_1\,t_2^{-1} \in T_{G(q)}$.  
\end{definition}
\begin{remark}\label{ZPlstd} The action of $T^m$ on $\mathcal{Z}_P$ is locally standard. The proof of this fact is analogous to the proof in Construction 5.12 in \cite{BP}.
\end{remark}
Note that the equivalence relation depends only on the combinatorics of $P^n$.  In fact, this is also true for the topological and smooth type of $\mathcal{Z}_P$, that is, combinatorially equivalent simple polytopes yield homeomorphic, and, in fact, diffeomorphic, moment angle manifolds (see Proposition 4.3 in Panov \cite{P} and the remark immediately following it).

The free action of $T^m$ on $T^{\mathcal{F}} \times P^n$ descends to an action on $\mathcal{Z}_P$, with quotient $P^n$.  
Let $\pi_{\mathcal{Z}}: \mathcal{Z}_P \longrightarrow P^n$ be the orbit map.  The action of $T^m$ on $\mathcal{Z}_P$ is free over the 
interior of $P^n$, where each vertex $v \in P^n$ represents the orbit $\pi_{\mathcal{Z}}^{-1}(v)$ with maximal isotropy subgroup of dimension $n$.  

In  \cite{BP} the following facts about the space $\mathcal{Z}_P$ are proven.

\begin{proposition}\label{p:BP} \cite{BP}  Let $P^n$ be a combinatorial simple polytope with $m$ facets, then 
\begin{enumerate}
\item  The space $\mathcal{Z}_P$ is a smooth manifold of dimension $m + n$. 
\item If $P = P_1 \times P_2$ for  
simple polytopes $P_1$ and $P_2$, then $\mathcal{Z}_P = \mathcal{Z}_{P_1} \times \mathcal{Z}_{P_2}$.  
          If $ G \subset P$ is a face, then $\mathcal{Z}_G$ is a submanifold of $\mathcal{Z}_P$.  
 \end{enumerate}
\end{proposition}


\subsection{Torus Orbifolds} 
In this subsection we gather some preliminary results about torus orbifolds.
We first recall the definition of an orbifold. For more details about orbifolds and actions of tori on orbifolds, see, for example, Haefliger and Salem \cite{HS1},  and \cite{GGKRW}.
\begin{definition}[{\bf Orbifold}] An {\em $n$-dimensional (smooth) orbifold}, denoted by $\mathcal{O}$, is a second-countable, Hausdorff topological space $|\mathcal{O}|$, called the underlying topological space of $\mathcal{O}$, together with an  equivalence class of $n$-dimensional orbifold atlases.
\end{definition}
In analogy with a torus manifold, we may define a torus orbifold, as follows. 

\begin{definition}[{\bf Torus Orbifold}] A {\em torus orbifold}, $\mathcal{O}$, is a $2n$-dimensional, closed, orientable orbifold with an effective smooth action of 
an $n$-dimensional torus $T$ such that $\mathcal{O}^T\neq \emptyset$.
\end{definition}

The following theorem from \cite{GGKRW}, proven using results obtained therein for  torus orbifolds,  is of use in the proof of Theorem \ref{t:thma}.
\begin{theorem}\cite{GGKRW}\label{REtorusorbifold} Let  $M$ be an $n$-dimensional, smooth, closed, simply connected, rationally elliptic manifold with an isotropy-maximal $T^k$-action. Then there is a product $\hat{P}$ of spheres of dimension $\geq 3$, a torus $\hat{T}$ acting linearly on $\hat{P}$, and an effective, linear action of $T^k$ on $\hat{M}=\hat{P}/\hat{T}$, such that there is a $T^k$-equivariant rational homotopy equivalence $M\simeq_{\qqq} \hat{M}$.
\end{theorem}


 \subsection{Alexandrov Geometry} Recall that a complete, locally compact, finite dimensional length space $(X,\mathrm{dist})$ with curvature bounded from below in the triangle comparison sense is an \emph{Alexandrov space} (see, for example, Burago, Burago, and Ivanov \cite{BBI}).
 When $M$ is a complete, connected Riemannian manifold and $G$ is a compact Lie group acting on $M$ by isometries, the orbit space $X=M/G$ is equipped with the orbital distance metric induced from $M$, that is, the distance between $\overline{p}$ and $\overline{q}$ in $X$ is the distance between the orbits $G(p)$ and $G(q)$ as subsets of $M$.  Additionally, if $M$ has sectional curvature bounded below, that is, $\sec M\geq k$, for some $k\in \mathbb{R}$, then the orbit space $X$ is an Alexandrov space with $\curv X \geq k$. 
 
The \emph{space of directions} of a general Alexandrov space at a point $x$ is
by definition the completion of the 
space of geodesic directions at $x$. In the case of orbit spaces $X=M/G$, the space of directions $\Sigma_{\overline{p}}X$ at a point $\overline{p}\in X$ consists of geodesic directions and is isometric to
$S^{\perp}_p/G_p$, where $S^{\perp}_p$ is the unit normal sphere to the orbit $G(p)$ at $p\in M$.


\subsection{Geometric results in the presence of a lower curvature bound}

Finally, we recall some general results about $G$-manifolds with non-negative and almost non-negative curvature which we use throughout.
As noted earlier, that a torus manifold is an example of an $S^1$-fixed point homogeneous manifold, indeed, of a nested $S^1$-fixed point homogeneous manifold. Closed, simply-connected, fixed point homogeneous manifolds of positive curvature were classified in Grove and Searle \cite{GS1}. More recently, the 
 following theorem by Spindeler,  \cite{Spi}, gives a characterization of non-negatively curved $G$-fixed point homogeneous manifolds.

\begin{theorem}\label{Spindeler} \cite{Spi} Assume that $G$ acts fixed point homogeneously  on a closed, non-negatively curved Riemannian manifold $M$. Let $F$ be a fixed point component of maximal dimension. Then there exists a smooth submanifold $N$ of $M$, without boundary, such that $M$ is diffeomorphic to the normal disk bundles $D(F)$ and $D(N)$ of $F$ and $N$ glued together along their common boundaries, that is, 
\bdm
M  = D(F) \cup_{\partial} D(N).
\edm
Further, $N$ is $G$-invariant and all points of $M\setminus \{F\cup N\}$ belong to principal $G$-orbits. 
\end{theorem}

 The following two facts from \cite{Spi}, for the case where $M$ is a torus manifold of non-negative curvature, are important for what follows.

\begin{proposition} \label{Spindelerprop} \cite{Spi} Let $M, N$ and $F$ be as in Theorem \ref{Spindeler} and assume that $M$ is a closed, simply connected 
torus manifold of non-negative curvature.   Then
            $N$ has codimension greater than or equal to $2$ and $F$ is simply connected.
  \end{proposition}

For a non-negatively curved torus manifold, Proposition 4.5 from  Wiemeler \cite{Wie} shows that the quotient space, $M^{2n}/T^n=P^n$, is described as follows:
 $P^n$ is a nice homology polytope  and  $P^n$ is of the form
\begin{equation}
\label{2.2}
P^n=\prod_{i<r} \Sigma^{n_i} \times \prod_{i\geq r} \Delta^{n_i},
\end{equation}
where $\Sigma^{n_i} =S^{2n_i}/T^n_i$ and $\Delta^{n_i}=S^{2n_i+1}/T^{n_i+1}$  is an $n_i$-simplex. The $T^{n_i}$-action on $S^{2n_i}$ is the suspension of the standard $T^{n_i}$-action on $\rrr^{2n_i}$, and it is easy to see that 
$\Sigma^{n_i}$ is obtained as the suspension of $\Delta^{n_i-1}$, ignoring the simplicial structure of $\Delta^{n_i-1}$. In what follows, we will refer to $\Sigma^{n_i}$ as a {\em lunar suspension of $\Delta^{n_i-1}$}.
Note that each 
$\Delta^{n_i}$ has $n_i+1$ facets and each $\Sigma^{n_i}$ 
 has $n_i$ facets, and so, the number of facets of $P^n$ 
 is bounded between $n$ 
and $2n$. 

Using this description of the quotient space the following equivariant classification theorem is obtained in \cite{Wie}.

\begin{theorem}\label{t:torus}   \cite{Wie} Let $M$ be a simply connected, non-negatively curved torus manifold. Then $M$ is equivariantly diffeomorphic to 
a quotient of  
\begin{equation}\label{2.3}
\mathcal{Z}_P=\prod_{i<r} S^{2n_i} \times \prod_{i\geq r} S^{2n_i-1},\, \, n_i\geq 2,
\end{equation} by a free linear torus action, where $\mathcal{Z}_P$ is the moment angle manifold corresponding to the polytope in Display \eqref{2.2}.
\end{theorem}

In the proof of Theorem \ref{t:torus}, the following lemma was important. It is also useful for the proof of Theorem \ref{t:thma}.
 
\begin{lemma}\label{locallystandard}  \cite{Wie} Let $M^{2n}$ be a simply connected torus manifold with an invariant metric of non-negative curvature. Then $M^{2n}$ is locally standard and $M^{2n}/T^n$ and all its faces are diffeomorphic (after smoothing the corners) to standard discs $D^k$. Moreover, $H^{\textrm{odd}}(M;Z) = 0$.
\end{lemma}

Using Proposition 4.5 of \cite{Wie} and Theorem 4.2 \cite{D}, the 
 following proposition  allows  us to identify the quotient space of a torus manifold, $M/T$.

\begin{proposition}\label{Q} Let $M^{2n}$ be a simply connected torus manifold with an invariant metric of non-negative curvature. Then $M/T=P$ is diffeomorphic to  a product of simplices and lunar suspensions   as in Display \eqref{2.2}.
\end{proposition}

We  also make use of the following theorem from \cite{GGKRW}. 
\begin{theorem}\cite{GGKRW}\label{misre} Let $M$ be a closed, simply connected, non-negatively curved Riemannian manifold admitting an effective, isometric, isotropy-maximal torus action. Then $M$ is rationally elliptic.
\end{theorem}

The following corollary of  Theorem \ref{misre} follows for an almost isotropy-maximal torus action via a simple adaptation of the proof in \cite{GGKRW} (cf. the proof of Theorem \ref{t:amism}).

\begin{corollary} Let $M$ be a closed, simply connected, non-negatively curved Riemannian manifold admitting an effective, isometric, almost isotropy-maximal torus action. Then $M$ is rationally elliptic.
\end{corollary}

In the proof of Theorem \ref{t:thma}, we will need to consider  generalizations of Theorem \ref{Spindeler}, Proposition \ref{Spindelerprop} and Lemma \ref{locallystandard} to manifolds of almost non-negative curvature.   We 
 recall the definition of almost non-negative curvature here, as well as an important result of Fukaya and Yamaguchi that allows us to determine under what conditions the total space of a principal torus bundle will admit a metric of almost non-negative curvature. 
\begin{definition}[{\bf Almost Non-Negative Curvature}]
A sequence of Riemannian manifolds $\left\{ \left( M,\mathrm{g}_{\alpha
}\right) \right\} _{\alpha =1}^{\infty }$ is almost non-negatively curved if
there is a real number $D>0$ so that \begin{eqnarray*}
\mathrm{Diam}\left( M,\mathrm{g}_{\alpha }\right) &\leq &D, \\
\mathrm{sec}\left( M,\mathrm{g}_{\alpha }\right) &\geq &-\frac{1}{\alpha }.
\end{eqnarray*}
\end{definition}

A large number of examples of almost non-negatively curved manifolds can be constructed by the following result of Fukaya and Yamaguchi \cite{FY}.
\begin{theorem}\label{FukayaYamaguchi}  \cite{FY}  Let $F \hookrightarrow E \longrightarrow B$ be a smooth fiber bundle with compact Lie structure group $G$,  such that $B$ admits a family of metrics with almost non-negative 
curvature and the fiber $F$ admits a $G$-invariant metric of non-negative curvature.  Then the total space $E$ admits a family of metrics with almost 
non-negative curvature. 
\end{theorem}


\section{The Equivariant Cross-Sectioning Theorem and the Equivariant Classification}\label{s3}

In this section we develop three tools which will help us to prove Theorem \ref{t:thma}. They are  the Cross-Sectioning Theorem \ref{t:ecst}, the
Equivariant Classification Theorem \ref{t:ect} and  Theorem  \ref{c:wpd}, which shows that a closed, simply-connected manifold, $M$, with a locally standard $T^k$-action such that $M/T=P$, as in Display \eqref{2.2}, has $k$ facets is equivariantly diffeomorphic to $\mathcal{Z}_P$, as in Display \eqref{2.3}.


  \subsection{Cross-Sectioning Theorem}

Before we state the Cross-Sectioning Theorem, we define a conical decomposition of a quotient space. 

 \begin{definition}[{\bf Conical Decomposition}] We say that the orbit space, $M^*=M^n/T^k$,  of $M^n$ by an effective $T^k$ action, admits a {\em conical decomposition} if 
we may decompose 
the orbit space $M^*$ into a collection of conical sections $\{C_i^*\}_{i=1}^{m}$, with $C_i^*\cong D^{n-k}$ for each $i\in \{1, \hdots, m\}$. 
Moreover, the $C^*_i$ satisfy the following property, namely, 
$$(C^*_1 \cup \cdots \cup C^*_j)\cap C^*_{j+1}=A_j,$$
where $A_j$ is an $(n-k-1)$-cell for each $j$, where $1\leq j\leq m$.
 \end{definition}

\begin{thmcst}\label{t:ecst} Let $T^k$ be a smooth, locally standard action  on a smooth, closed, $n$-dimensional manifold, $M^n$ such that $M^*=M^n/T^k$ is homeomorphic to an $(n-k)$-dimensional disk, $D^{n-k}$, and $M^*$ admits a conical decomposition. Then there exists a continuous cross-section to the orbit map $\pi: M^n\rightarrow M^*$ that is smooth on a closed subset of $M^*\setminus \partial M^*$.
\end{thmcst}

The proof generalizes elements of the equivariant  classification theorem for $T^3$-actions on $M^6$ in \cite{Mc}.  Similar techniques are also used in 
 work of Raymond \cite{R} and Orlik and Raymond \cite{OR1, OR2, OR3}.  
 The main topological tool in the proof comes from obstruction theory, that is, the obstruction to extending a map $A \longrightarrow X$ to a map $W \longrightarrow X$, where 
 $X$ is a connected CW-complex and $(W,A)$ is a CW-pair.  Such an extension always exists, that is, the obstruction vanishes, if $H^{n+1}(W,A, \pi_n(X)) = 0$ for all $n$.  For 
 more details on obstruction theory see, for example, Davis and Kirk \cite{DK}.

 In the following two lemmas, required for the proof of the Cross-Sectioning Theorem \ref{t:ecst}, we will be considering a closed $T^k=T_1\times \cdots \times T_k$-invariant subset $C$ of the closed manifold, $M$, with $M$ as in Theorem \ref{t:ecst}. We choose $C$ so that its orbit space under the torus action, $C^*\subset M^*=D^{n-k}$  is a closed conical section  of $M^*$, with conical orbit space, as in Definition \ref{defncone}. That is, 
 $C^*\cong D^{n-k}=K(D^{n-k-1}),$
 is a closed cone over $D^{n-k-1}$. 
 Moreover,  we choose $C$ so that its intersection with the boundary, $\partial M^*$, is homeomorphic to $D^{n-k-1}$. In Figure \ref{F:1}, we illustrate this orbit space $C^*$  and its homeomorphic image $D^{n-k}$. 
We decompose $\partial C^*$ into its upper and lower hemispheres, $S^+$ and $S^{-}$, respectively, where $S^+= \partial C^*\cap \partial M^*$, $S^+\cap S^{-}=S^{n-k-2}$, and $p^*\in S^+$ denotes the vertex, as in Definition \ref{defncone}.

  \begin{figure}[!h]
\centering
\psfrag{C^*}{(1)}
\psfrag{$D^{n-k}$}{(2)}
\includegraphics[scale=0.8]{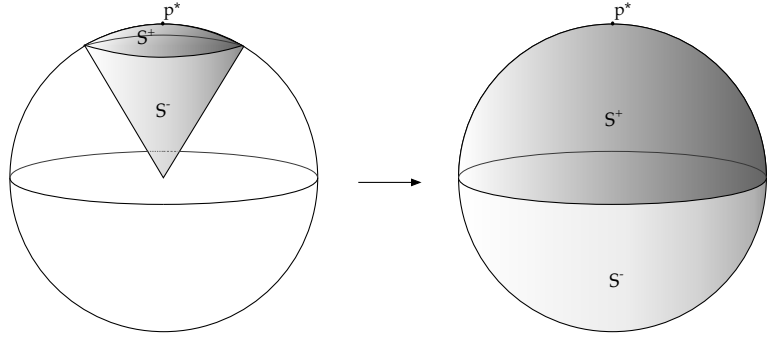}
\caption{The conical section $C^*\subset M^*$ and its homeomorphic image $D^{n-k}$.}
\label{F:1}
\end{figure}

Let $i$, $1\leq i\leq \min(n-k, k)$, denote the rank of the isotropy subgroup corresponding to the point $p^*$.  Since isotropies are constant along rays from $p^*$,   we can partition $(S^+, p^*)$ into $i$ cells of  dimension $(n-k-1)$, denoted by  $U_{k-i+1}, \hdots, U_k$, provided they all intersect in $p^*$, that is, $p^*\in \bigcap_{j=k-i+1}^{k} U_j$.  Note that to each $U_l$ we associate the corresponding circle isotropy subgroup $T_l$, where $k-i+1\leq l\leq k$.
By assumption,  the $T_l$ generate the $i$-dimensional torus $T^i$, $1\leq i\leq \min(n-k, k)$, and each pair of distinct circles has trivial intersection, that is, $T_{{k-i+1}}\times\cdots\times T_{k}=T^i$.

This gives us a weighted decomposition of $(S^+, p^*)$, which we denote  by 
$$ \{(U_{k-i+1}, T_{{k-i+1}}), \hdots, (U_k,T_{k})\}.$$
It is understood then in this decomposition that each intersection of $j$-cells in a $(j-1)$-cell corresponds to the connected isotropy subgroup of the $T^k$-action generated by 
the isotropy subgroups associated to each of the $j$-cells.

The simplest possible decomposition is as in Lemma \ref{l:cs1}, where the decomposition of $(S^+, p^*)$ is given as $\{(U_k, T_k)\}$ and is illustrated in the right hand figure of Figure \ref{F:1}. The next simplest is given as $\{(U_{k-1},  T_{k-1}) (U_k, T_k)\}$ and is illustrated in Figure \ref{F:2}.
The most general decomposition will be the one where $p^*$ corresponds to an orbit with $T^i$ isotropy, and whose weighted decomposition is  $ \{(U_{k-i+1}, T_{{k-i+1}}), \hdots, (U_k,T_{k})\}$.
Note that 
all non-trivial isotropies are connected and correspond to points on $S^+$, all other orbits are principal.

 \begin{figure}[!b]
\centering
\includegraphics[scale=0.8]{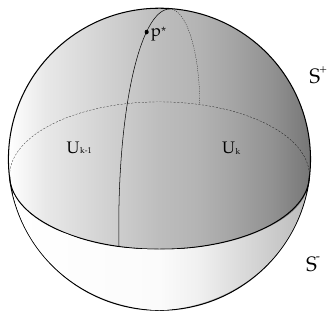}
\caption{The upper hemisphere  $S^+ = U_{k-1} \cup U_k$.}
\label{F:2}
\end{figure}

In the following lemma, we begin with the simplest  case and show that we can construct a cross-section for $C^*$.
 
 \begin{lemma}  \label{l:cs1}  Let $T^k = T_1 \times \dots \times T_k$ be a smooth, locally standard action on a smooth, closed $n$-dimensional subspace $C\subset M$, where $M$ is a smooth, closed $n$-dimensional closed manifold, and $C$ has quotient space, $C^*$,  as described above. 
 Suppose  $(S^+, p^*)=\{(U_k, T_k)\}$.
 Then the following hold:
 \begin{enumerate}
 \item There exists a cross-section from $C^*$ to $C$ that is smooth on a closed subset of $C^*\setminus \partial C^*$; and

 \item If a cross-section is given on an $(n-k-1)$-cell $A \subseteq S^-\subset C^*$,  it can be extended to all of $C^*$.  
 \end{enumerate}
  \end{lemma}

\begin{proof} To prove Part 1, 
recall that the orbit space, $C^*$, is a closed cone with vertex $p^*$, the north pole of $D^{n-k}$.  Since  the orbit structure is conical and $G_p=T_k$, 
by Theorem \ref{cone}, $C$ is equivariantly homeomorphic to $T^{k-1} \times D^{n-k+1}$ with $T_k$ acting orthogonally on $ D^{n-k+1}$.  Since the $T^k$ action is locally standard, we can 
construct a cross-section from $D^{n-k}$ to $D^{n-k+1}$, see  Remark \ref{conethm}.  
We then compose this section with a section from $D^{n-k+1}$ to $T^{k-1}\times D^{n-k+1}$ by sending an arbitrary point $x\in D^{n-k+1}$ to $(t, x)\in T^{k-1}\times D^{n-k+1}$ for some $t\in T^{k-1}$.  This gives us a cross-section from $C^*$ to $C$. We denote it by $s_1$.

We note that since $C^*\setminus \partial C^*$ is homeomorphic to an open $(n-k)$-disk and corresponds entirely to principal orbits, its inverse image in $C$ is the trivial bundle $T^k\times D^{n-k}$, and so there exists a smooth cross-section on $C^*\setminus \partial C^*$, which we will denote by $s_2$. 

Since $C^*$ is a cone over $\partial C^*$, we may write 
$C^*=K(\partial C^*)=(\partial C^*\times [0, 1])/(\partial C^*\times \{0\})$ 
and define the following straight line homotopy between $s_2$ and $s_1$ over 
$\partial C^*\times [\frac{1}{2}, 1]\subset K(\partial C^*)$:

$$h(x, t)=(2t-1)s_1(x)+ (2-2t)s_2(x), \forall x\in \partial C^*\times [\frac{1}{2}, 1].$$
We may then define a cross-section $s(x, t)$ on $C^*=K(\partial C^*)$ as
$$s(x, t)=\begin{cases} 
s_2(x)    & t\in [0, \frac{1}{2}]\\
h(x, t) &t\in[\frac{1}{2}, 1]\\

\end{cases}\,\,,$$
noting that the cross-section $s(x, t)$ is smooth on a closed subset of $C^*\setminus \partial C^*$, as desired.   Note that this homotopy can be modified to make the closed subset as  close to the boundary as we like.

    To prove Part 2, suppose a cross-section $s$ is given on a $(n-k-1)$-cell $A \subseteq S^{-} \subset C^* $.  Let $A' = A \cap (D^{n-k} \,\setminus \, S^{+})$ and 
$\pi: C \longrightarrow C^{*} \cong D^{n-k}$ the orbit map.  Then $\pi^{-1}(D^{n-k} \,\setminus \,S^{+})$ is a principal $T^{n-k}$-bundle over $(D^{n-k} \,\setminus \,S^{+})$.
Using the long exact sequence for relative cohomology and excision, it follows that $H^i((D^{n-k} \,\setminus \,S^{+}), A') = 0$ for all $i >0$. Thus, by obstruction theory, we may assume that $s$ is defined on $(D^{n-k} \,\setminus\,S^{+}) \cup A$.  

We now obtain the  diagram below,
where  $\pi_1:C\rightarrow \bar{C}_1=C/T^{k-1}$ and
$\pi_2:\bar{C}_1\rightarrow C^{*} \cong D^{n-k}$.

\begin{equation*}
\begin{split}
\xymatrix{
& {} &  C   \ar[d]^{\pi} \ar[r]^{\pi_1}   \ar[d]^{\pi}  &  {\bar{C}}_1    \ar[ld]^{\pi_2} 
  &   &\\
& {} (D^{n-k} \,\setminus\,S^{+}) \cup A \ar[ru]^{s}\ar[r]  &D^{n-k} &&} 
\end{split}
\end{equation*}

Let 
$s_2 = \pi_1 \circ s$.  Then $s_2$ is a cross-section to $\pi_2$ defined on $(D^{n-k}\, \setminus\, S^{+}) \cup A$.  But $S^{+}$ corresponds to the set of fixed points
of the $T_k$-action on ${\bar{C}}_1$, so we can define $s_2$ on all of $D^{n-k}$.   In order to show continuity of $s_2$ we first describe $D^{n-k}$ as $I \times S^{+}$, 
where $I$ is an interval.   Note that $\pi_2^{-1}(S^{+}) \cong D^{n-k-1}$ and 
$\pi_2^{-1}(I) \cong K(T_k) \cong D^2$ and the $T_k$-action on $\pi_2^{-1}(I) $ is rotation.  Hence the $T_k$-action on $C_1 \cong D^{n-k+1} \cong D^2 \times D^{n-k-1}$ is 
rotation on the first factor and trivial on the second.  An orbit of the $T_k$-action can be described as $\{(r e^{i \theta}, R e^{i \Theta}) | 0 \le \theta < 2\pi, 0 \le \Theta < 2 \pi \}$ where 
\[
R e^{i \Theta} =\begin{cases}
 (r_1 e^{i \theta_1}, \dots, r_{{n-k-1}\over 2} e^{i \theta_{{n-k-1}\over 2}}) &\textrm{ if $n-k-1$ is even, } \\
 (r_1 e^{i \theta_1}, \dots, r_{{n-k-2}\over 2} e^{i \theta_{{n-k-2}\over 2}},1) & \textrm{ if $n-k-1$ is odd. } 
\end{cases}
\]
 Note that the fixed point set of $T_k$ on $D^2 \times D^{n-k-1}$ is 
$\{0\} \times D^{n-k-1}$.  Now let $q=(0,S e^{i\Sigma}) \in \text{Fix}(T_k, D^2 \times D^{n-k-1})$ and let $\{q_n^*\} = \{(r_ne^{i\theta_n}, R_n e^{i\Theta_n})^*\}$ be a sequence 
in $M^*$ converging to $q^*$.  Then $r_n \longrightarrow 0, R_n \longrightarrow S$ and $\Theta_n \longrightarrow \Sigma$.  But then the sequence $\{s_2(q_n^*)\}$ 
will be of the form $\{(r_n e^{i\theta_n}, R_n e^{i \Theta_n})\}$, which converges to $q=(0,S e^{i \Sigma})$.  Hence $s_2$ is continuous.  

 Next we define a cross-section $s_1$  on $s_2(C^*) \cong D^{n-k}$.  Let
$s_1 = s \circ \pi_2: s_2(C^*) \setminus s_2(S^+ \,\setminus\, A) \rightarrow C$.  Now $\pi_1^{-1}(s_2(C^*))$ is a principal $T^{k-1}$-bundle over $s_2(C^*) \cong D^{n-k}$ and we obtain that
 $s_1$ is in fact a cross-section of $\pi_1$ defined on $s_2(C^*) \setminus s_2(S^+ \,\setminus \,A) $.  Since $s_2(S^+ \,\setminus \,A) $ is a homology $(n-k-1)$-cell on the boundary of 
 $s_2(C^*)$, it follows that $H^i(s_2(C^*) , s_2(C^*) \setminus s_2(S^+ \,\setminus \,A)) = 0$ for all $i > 0$.  Again by obstruction theory this implies that $s_1$ can be extended
 to all of $s_2(C^*)$.  Thus $s_1 \circ s_2$ extends $s$ to all of $C^*\cong D^{n-k}$.  
 The diagram below illustrates this case.
   \begin{equation*}
\begin{split}
\xymatrix{
& {} &  C   \ar[r]^{\pi_1}   \ar[d]^{\pi}  &  {\bar{C}}_1    \ar[ld]^{\pi_2} \ar@/_2pc/[l]_{s_1}   &    \\
& {} &  D^{n-k}  \ar@/^+2pc/[u]_{s}  \ar@/_2pc/[ur]_{s_2}  & {}   }
\end{split}
\end{equation*} 
\end{proof}

We are now ready to construct a cross-section for a general decomposition of $C^*$. 

\begin{lemma}\label{cs4} Let $T^k = T_1 \times \dots \times T_k$ be a smooth, locally standard action on a smooth, closed $n$-dimensional subspace $C\subset M$, where $M$ is a smooth, closed $n$-dimensional closed manifold, and $C$ has quotient space, $C^*$, as described above. 
 Let the decomposition of $(S^+, p^*)$ be given by $\{(U_{k-i+1}, T_{{k-i+1}}), \hdots, (U_k,T_{k})\}$, with $1< i\leq n-k\leq k$.  Then the following hold: 
 \begin{enumerate}
 \item
 There exists a cross-section  from $C^*$ to $C$ that is smooth on a closed subset of $C^*\setminus \partial C^*$; and 
 \item If a cross-section is given on an $(n-k-1)$-cell $A \subseteq S^-$, then it can be extended to all of $C^*$.  
   \end{enumerate}
 \end{lemma}
  \begin{proof} To prove Part 1, 
we suppose that  $\{\{U_{k-i+1}, \hdots, U_k\}, T^i=T_{k-i+1}\times \cdots\times T_{k}\}$ is the decomposition for $(S^+, p^*)$.  The orbit structure is conical with  vertex $p_{*}$ where $G_p = T_{k-i+1} \times\cdots \times  T_k$.  By Theorem \ref{cone}, $C$ is equivariantly homeomorphic to $T_{1} \times \cdots {T_{k-i}} \times   D^{n-k+i} \cong T^{k-i} \times D^{n-k+i}$ and $T_{k-i+1}\times \cdots\times T_{k}\cong T^i$ acts orthogonally on $D^{n-k+i}$. It is easy to construct a cross-section in this case, following the proof of Lemma \ref{l:cs1}. Again, since the $T^k$ action is locally standard, we can 
construct a cross-section from $D^{n-k}$ to $D^{n-k+i}$, see  Remark \ref{conethm}.  Then we construct a section from $D^{n-k+i}$ to $T^{k-i}\times D^{n-k+i}$ by sending an arbitrary point $x\in D^{n-k-i}$ to $(t, x)\in T^{k-i}\times D^{n-k+i}$ for some $t\in T^{k-i}$. 

Using the same homotopy argument as in the proof of Part 1 of Lemma \ref{l:cs1}, 
we can then construct a  cross-section  from $C^*$ to $C$ that is smooth on a closed subset of $C^*\setminus \partial C^*$, as desired.

To prove Part 2, suppose a cross-section $s$ is given on a $(n-k-1)$-cell $A \subseteq S^{-}$.  As In Lemma \ref{l:cs1} we may assume 
that $s$ is defined on $ (D^{n-k} \setminus S^{+}) \cup A$.  We now obtain the  diagram below,  
where $\pi_1:C\rightarrow \bar{C}_1=C/T^{k-i}$, $\pi_j:\bar{C}_{j-1}\rightarrow \bar{C}_j=\bar{C}_{j-1}/T_{k-i-j+2}$, for $j\in\{2, \hdots, i\}$ and $\pi_i:\bar{C}_i\rightarrow C^{*} \cong D^{n-k}$.

 \begin{equation*}
\begin{split}
\xymatrix{
& {} &  C   \ar[rr]_{\pi_1}^{/T^{k-i}}  \ar[ddd]^{\pi}  & {} &  \bar{C}_1      
 \ar[d]_{\pi_2}    &    \\
& {} &  {}      & {} &  \vdots   \ar[d]_{\pi_{i}} 
 \\
& {} & {}    & {} &  \bar{C}_{i}   \ar[lld]_{\pi_{i+1}}  
 \\
&  (D^{n-k} \setminus S^{+}) \cup A\ar[ruuu]^{s}
\ar[r]    & D^{n-k}   
\\
}
\end{split}
\end{equation*}

Let $s_{i+1} = \pi_{i} \circ \pi_{i-1} \circ \dots \circ \pi_1\circ  s$.  Then $s_{i+1}$  is a cross-section to $\pi_{i+1}$ defined on $(D^{n-k} \setminus S^{+}) \cup A$. 
Now $\pi_{i+1}^{-1}(D^{n-k} \setminus U_k)$ is a principal circle bundle over $D^{n-k} \setminus U_k$, and since $H^q(D^{n-k} \setminus U_k, (D^{n-k} \setminus S^{+}) \cup A) = 0$
for all $ q > 0$, $s_{i+1}$ can be extended to all of $D^{n-k} \setminus U_k$.  But $U_k$  corresponds to the set of fixed points
of the $T_k$ action on $\bar{C}_{i}$, so  $s_{i+1}$ can be extended continuously to all of $D^{n-k}$, exactly as in the proof of  Lemma \ref{l:cs1}.

We assume then that we have a cross-section to $\pi_{l+1}$ defined on $s_{l+2}\circ\cdots \circ s_{i+1}(D^{n-k})\cong s_{l+2}(D^{n-k})$.
Next we need a cross-section to $\pi_l$ defined for any $l$, $2\leq l\leq i$  and on $s_{l+1}(D^{n-k})$.
Let
$$s_l= \pi_{l-1}\circ \cdots \circ \pi_1\circ s\circ \pi_{i+1} \circ\cdots \pi_{l+1}.$$
Again, $\pi_{i}^{-1}(s_{l+1}(D^{n-k}) \setminus s_{l+1}(U_{k-l}))$ is a principal circle bundle over $s_{l+1}(D^{n-k}) \setminus s_{l+1}(U_{k-l})$ and 
$s_{l}$ can be extended to all of $s_{l+1}(D^{n-k}) \setminus s_{l+1}(U_{k-l})$.  Then it follows that $s_{l+1}(D^{n-k}) \cong D^{n-k}, s_{l+1}(U_{k-l}) \cong D^{n-k-1} \cong 
\partial(s_{l+1}(D^{n-k}))$, and  $\pi_{l}^{-1}(s_{l+1}(U_{k-l}))$ is the fixed point set of the $T_{k-l}$ action on  $\pi_{l}^{-1}(s_{l+1}(D^{n-k}))$.  Hence 
exactly as before we may extend $s_{l}$ continuously to all of $s_{l+1}(D^{n-k})$.

We now arrive at  $\pi_1^{-1} ( (s_2 \circ s_3 \circ \dots \circ s_{i+1})(D^{n-k}))$ is a principal $T^{k-i}$-bundle over 
$(s_2 \circ s_3 \circ \dots \circ s_{i+1})(D^{n-k}) \cong D^{n-k}$ and $s_1 = s \circ \pi_{i+1} \circ \dots \circ \pi_2$ is a cross-section defined on 
$(s_2 \circ s_3 \circ \dots \circ s_{i+1})((D^{n-k} \setminus S^{+}) \cup A).$  As before this can be extended to all of $(s_2 \circ s_3 \circ \dots \circ s_{i+1})(D^{n-k})$.  
Hence $s_1 \circ s_2 \circ s_3 \circ \dots \circ s_{i+1}$ is an extension of $s$ to all of $C^*$.  
The diagram below illustrates this case.
 \begin{equation*}
\begin{split}
\xymatrix{
& {} &  C   \ar[rr]_{\pi_1}^{/T^{k-i}}  \ar[ddd]^{\pi}  & {} &  \bar{C}_1   \ar@/_2pc/[ll]_{s_1}     \ar[d]_{\pi_2}    &    \\
& {} &  {}      & {} &  \vdots   \ar[d]_{\pi_{i}}  \ar@/_1pc/[u]_{s_2}  \\
& {} & {}    & {} &  \bar{C}_{i}   \ar[lld]_{\pi_{i+1}}  \ar@/_1pc/[u]_{s_i}  \\
& {} & D^{n-k}  \ar@/^+2pc/[uuu]_{s}    \ar@/_1.5pc/[urr]_{s_{i+1}} \\
}
\end{split}
\end{equation*}

\end{proof}

 We are now ready to prove the Cross-Sectioning Theorem \ref{t:ecst}.

\begin{proof}[Proof of  the Cross-Sectioning Theorem \ref{t:ecst}]  First, since the quotient space admits a conical decomposition, we decompose 
the orbit space $M^* \cong D^{n-k}$ into a collection of conical sections $\{C_i^*\}_{i=1}^{m}$, with $C_i^*\cong D^{n-k}$ for each $i\in \{1, \hdots, m\}$, and such that the orbit structure for each $C_i \subset M^n$ is conical as in Definition \ref{defncone}. 
As we saw in the proof of Lemma \ref{cs4}, each $C_i \cong T^{k-l} \times D^{n-k+l}$ for some $l\in\{1,\dots k\}$ and therefore a cross-section 
exists on each $C_i^*$.   

In order to create a cross-section on all of $M^*$, 
we start by defining a cross-section $s$ on $C_1^*$.  We then attach $C_2^*$ to $C_1^*$ along an $(n-k-1)$-cell 
 $A_1$ and so we have a cross-section $s$ defined on $C^*_1$ and on $A_1 \subset S^{-} \subset C_2^*$.  By Lemma \ref{cs4}, we can then  extend the cross-section $s$ to all of $C_2^*$.   Continuing this process, we attach each additional $C_i^{*}$ along an $(n-k-1)$-cell $A \subset S^{-} \subset C_i^*$.  Therefore $s$ can be extended to all of  $C_i^{*}$.  It follows that we can extend $s$ to all of $M^{*}$. Figure \ref{F:4} illustrates this process for a partial decomposition of $M^*$. 
 
 Using the same homotopy argument as in the proofs of Lemmas \ref{l:cs1} and \ref{cs4}, 
we can then construct a  cross-section  from $M^*$ to $M$ that is smooth on a closed subset of $M^*\setminus \partial M^*$, as desired.

\end{proof}

 \begin{figure}[!b]
\centering
\includegraphics[scale=0.8]{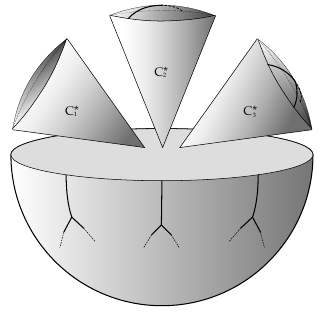}
\caption{A partial decomposition of the quotient space, $M^*$, highlighting three conical sections in  the upper hemisphere of $M^*$. The tree-like lines on $M^*$ represent higher rank isotropy groups.}
\label{F:4}
\end{figure}


 \subsection{Equivariant Classification Theorem}

For classification purposes it is convenient to fix orientations.  We start with a fixed orientation of the group $G$.  Then an orientation of $M$ determines an orientation of 
$M^*$ and vice versa, assuming there are no isotropy groups which reverse the orientation of a slice.  When the orbit map $\pi:M\rightarrow M^*$ has a {\em cross-section $s$}, that is, $s: M^* \rightarrow M$ is a continuous map such that $\pi\circ s$ is the identity on $M^*$, we always assume that 
the orientation of $s(M^*)$ is the one induced by the cross-section and the orientation of $M$ is compatible with it.

 We begin with a generalization of the equivariant homeomorphism theorems of \cite{OR3} and \cite{O}.

\begin{theorem}\label{t:homeoect}  Let $G$, a compact Lie group,  act continuously on  closed, $n$-dimensional manifolds, $M_1^n$ and  $M^n_2$  such that the orbit maps $\pi_i: M_i^n\rightarrow M_i^*$ are continuous  and have continuous cross-sections $s_i:  M_i^*\rightarrow M_i^n$, $i=1, 2$. 
 Then there exists a $G$-equivariant homeomorphism $h$ from  $M^n_1$ onto  $M^n_2$ if and only if there exists a weight preserving homeomorphism $h^*$ from $M^*_1$ onto  $M^*_2$.  
 
 Furthermore, if $M_1$ and $M_2$ are oriented and the orientations of $M_1^*$ and $M_2^*$ are those induced by 
$M_1$ and $M_2$, then $h$ is orientation preserving if and only if $h^*$ is orientation preserving.  
\end{theorem}
\begin{proof}
Let $M^n_1$ and $M^n_2$ be closed $G$-manifolds such that their respective orbit maps $\pi_i$, $i=1, 2$ admit continuous cross-sections $s_i$, $i=1, 2$. Denote their quotients as $M^*_1$ and $M^*_2$, respectively. 
Recall that a weight preserving homeomorphism is a continuous strata-preserving map with continuous inverse that carries the weights of $M^*_1$ isomorphically to those of $M^*_2$. 
So, if $h: M^n_1\longrightarrow M^n_2$ is a $G$-equivariant homeomorphism, then define $h^*:=\pi_2\circ h\circ s_1$. Since $h^*$ is a composition of continuous maps that is $1$-to-$1$ and onto, with continuous inverse, $(h^*)^{-1}:=\pi_1\circ h\circ s_2$, $h^*$ is the desired   weight preserving homeomorphism from $M^*_1$ to $M^*_2$.

Vice versa, if we assume that $h^*: M^*_1 \longrightarrow M^*_2$ is a weight preserving homeomorphism, then  both the composition $s_2\circ h^*\circ \pi_1$
 and the inverse of this composition, $\pi_2\circ (h^*)^{-1}\circ s_1$, are clearly continuous.  For $x\in s_1(M_1^*)$, we  define $h: M^n_1\rightarrow M^n_2$ so that $h(x)=s_2\circ h^*\circ \pi_1(x)$. Then for any $y\in M^n_1$, since $y=gx$ for some $g\in G$ and some $x\in s_1(M_1^*)$, we define  $h(y)=gh(x)$.  This gives us the desired $G$-equivariant homeomorphism.
 \end{proof}

We now recall a corollary of the  Whitney Approximation Theorem (see also Corollary 6.27 in Lee  \cite{Lee}), which states the following.

\begin{corollary}\label{extension}\cite{Lee}  Suppose $N$ is a smooth manifold with or without boundary, $M$ is a smooth manifold without boundary, $A\subset N$ is a closed subset and $f:A\rightarrow M$ is a smooth map. Then $f$ has a smooth extension to $N$ if and only if it has a continuous extension to $N$.
\end{corollary}

Using Corollary \ref{extension}, we can now extend the ``if" statement of Theorem \ref{t:homeoect} to equivariant diffeomorphism. 
\begin{ectthm}\label{t:ect}  Let $G$, a compact Lie group,  act smoothly on smooth, closed, $n$-dimensional manifolds, $M_1^n$ and  $M^n_2$ whose quotient spaces $M^*_1$ and $M^*_2$ are manifolds with corners. Moreover, assume that for $i=1, 2$, the orbit maps $\pi_i: M_i^n\rightarrow M_i^*$ are  smooth  and have continuous cross-sections $s_i:  M_i^*\rightarrow M_i^n$,  that are smooth on the closed subsets $N_1^*\subset M_1^*\setminus \partial M_1^*$ and $N_2^*=h^*(N_1^*)$. 
 If  there exists a weight preserving diffeomorphism $h^*$ from $M^*_1$ onto  $M^*_2$, then there exists a $G$-equivariant diffeomorphism $h$ from  $M^n_1$ onto  $M^n_2$.  
 
 Furthermore, if $M_1$ and $M_2$ are oriented and the orientations of $M_1^*$ and $M_2^*$ are those induced by 
$M_1$ and $M_2$,   then $h$ is orientation preserving if $h^*$ is orientation preserving.  
\end{ectthm}
\begin{proof}
Let $M^n_1$ and $M^n_2$ be closed $G$-manifolds such that their respective orbit maps $\pi_i$, $i=1, 2$ admit continuous cross-sections $s_i$, $i=1, 2$, and  their respective restrictions to $N_1^*$ and $N_2^*$ are smooth. Denote their quotients as $M^*_1$ and $M^*_2$, respectively.   
Recall that a weight preserving diffeomorphism is a weakly smooth, strata-preserving map with weakly smooth inverse that carries the weights of $M^*_1$ isomorphically to those of $M^*_2$.

By assumption, there exists a continuous cross-section to the orbit map $\pi_i: M_i^n\rightarrow M_i^*$ that is smooth on the closed subset  $N_i^*$, for $i=1, 2$. 
Since the $\pi_i$ are smooth, they are continuous,  so we have that $N_i=\pi_1^{-1}(N_i^*)$ are closed subsets of $M_i$, for $i=1, 2$. 
Since $h^*: M^*_1 \longrightarrow M^*_2$ is a weight preserving diffeomorphism, then  both the composition $s_2\circ h^*\circ \pi_1$
 and the inverse of this composition, $s_1\circ (h^*)^{-1}\circ \pi_2$, are clearly continuous and are simultaneously smooth when restricted to  $N_1$ and $N_2$, respectively.  Let $x\in s_1(M_1^*)$ and define $h(x)=s_2\circ h^*\circ \pi_1(x)$. Since  for any $y\in M_1$, we have $y=gx$ for some $g\in G$,  
 we then define $h: M_1\rightarrow M_2$ by $h(y)=gh(x)$.

We may now  apply 
 Corollary \ref{extension} to obtain that $h:M_1\rightarrow M_2$ is smooth. A similar argument works to show that $h^{-1}$ is smooth. Thus $h$ is a $G$-equivariant diffeomorphism from $M_1$ to $M_2$.
 \end{proof}

\begin{remark} 
 The ``only if" statement of Theorem \ref{t:homeoect} cannot be extended to the smooth setting with the techniques of the proof of Theorem \ref{t:ect}, as  Corollary \ref{extension} is, in general, false  if the target manifold, $M$,  has boundary (see Problem 6-7 in \cite{Lee}).  \end{remark}
 

\subsection{Creating a weight preserving diffeomorphism}

In this subsection we prove Theorem \ref{c:wpd}, which is an essential component of the proof of Theorem \ref{t:thma}.
We begin with the following extension of work of Oh \cite{Oh2}. 

\begin{lemma}\label{ohthesis} Let $T^k$ act smoothly and effectively on $M^n$, a closed, manifold with $M^n/T^k=D^{n-k}$ such that all isotropy subgroups are connected and all singular orbits correspond to points on the boundary of the quotient space. Then $\pi_1(M)$ is trivial if and only if the isotropy subgroups of the $T^k$-action on $M^n$ generate  $T^k$.
\end{lemma}
\begin{proof}
The ``if" portion of the statement is a straightforward generalization of the proof of Corollary 1.2 in \cite{Oh2}, so we prove the ``only if" statement.
The proof generalizes the proof of Theorem 1.1 in \cite{Oh2}. 
By assumption, $k\leq n-1$.
If $\alpha$ is an element of $\pi_1(M)$, then by the Whitney embedding theorem, there is an embedding $f:S^1\rightarrow M$ which represents $\alpha$. By transversality, $f$ is homotopic to $g: S^1\rightarrow M_{\mathrm{reg}}$, where $M_{\mathrm{reg}}$ is the regular part of $M$, that is, the union of all principal $T^k$-orbits.

Thus $\iota_*:\pi_1(D^{n-k}\times T^k)=\pi_1(M_{\mathrm{reg}}) \rightarrow \pi_1(M)$ is a surjection, where $\iota_*$ is the homomorphism induced by the inclusion. Since the isotropy subgroups $H(\mathbf a_1), \hdots, H(\mathbf a_k)$ span $T^k$ by Corollary \ref{c:KMP}, the determinant of the corresponding matrix is nonzero.

Let $[H_i x]$ be the homotopy class  represented by the circle $H_i x, x\in M_{\mathrm{reg}}$.  Since $H_i$ is the isotropy of a facet of $M/T$, it follows that $H_i x$ bounds a disk in $M$. So $[H_i x]\in \ker(\iota_*)$.

Hence $\pi_1(M)\cong \mathbb{Z}^k/\ker(\iota_*)\subseteq \mathbb{Z}^k/\langle [H_1 x], \hdots, [H_k x]\rangle$. Since we assume $\pi_1(M)$ is trivial,  it is immediate that the determinant of the matrix formed by the $H(\mathbf a_1), \hdots, H(\mathbf a_k)$ is $\pm 1$ and the result follows by Lemma \ref{oh2}.
\end{proof}

 With the hypotheses as in Lemma \ref{ohthesis}, 
Corollary \ref{c:KMP} tells us that the quotient space, $M/T$, has at least $k$ facets and Lemmas \ref{oh2} and \ref{ohthesis} tell us that there is a $(k\times k)$ sub-matrix of weights for $M/T$ that must have determinant $\pm 1$.

\begin{lemma}\label{Smith} Let  $T^k$ act isometrically, effectively on  a closed, simply connected, $n$-dimensional Riemannian manifold, 
whose quotient space $M/T$ is   
\begin{equation*}
P^{n-k}=\prod_{i<r} \Sigma^{n_i} \times \prod_{i\geq r} \Delta^{n_i},
\end{equation*}
and the number of facets of $P^{n-k}$ is equal to $k$. Suppose that all singular isotropies are connected and correspond to points on the boundary and all interior points correspond to principal orbits. Then there exists a weight preserving diffeomorphism $\phi$ of  $P^{n-k}$  
taking the weight vectors, 
$\{ {\bf{a}}_i\}$, of $P^{n-k}$ to $\{ \phi({\bf{a}}_i)\}$, where $\phi({\bf{a}}_i)$ is the unit vector with the value $\pm 1$ in the $i$-th position.
\end{lemma}

\begin{proof}
To prove the lemma, all we need to do is exhibit the isomorphism of the weights, as $\phi|_{P^{n-k}}$ can be taken to be the identity. Since $P^{n-k}$ is homeomorphic to an $(n-k)$-disk,  it follows by Corollary \ref{c:KMP} that the circle isotropy subgroups span $T^k$, and by Lemma \ref{ohthesis} they must generate $T^k$.
We then may assign a weight to each facet, corresponding to the circle isotropy subgroup of $T^{k}$. Each weight may be written as a $k$-vector  and we may group all the weights in a $k\times k$ matrix.  By Lemma \ref{oh2},  the determinant of this matrix is equal to $\pm 1$. 
Since the $k\times k$ matrix of the weights has integer entries and non-zero determinant, using the Smith normal form, it can be made diagonal. 
Once it is in diagonal form,
we obtain a $k\times k$ matrix with $\pm 1$ entries on the diagonal, as desired.  
\end{proof}

\begin{lemma}\label{Free} With the same hypotheses as in Lemma \ref{Smith},  
suppose that the matrix of weights of $P^{n-k}$ is a matrix with $\pm 1$ entries along the diagonal. Then the $T^k$-action has $(2k-n)$ freely acting circles.
\end{lemma}

\begin{proof}

By assumption the matrix of weights of $P^{n-k}$ is a matrix with $\pm 1$ entries along the diagonal. Hence all circle isotropy subgroups are mutually orthogonal.
In order to show that the $T^{k}$-action has $(2k-n)$ freely acting circles, we claim that it suffices to find $(2k-n)$ pairs of opposing facets on the polytope, that is $(n-k-1)$-dimensional faces that do not intersect in any lower-dimensional faces. For each such pair of opposing facets, we consider the diagonal circle in the subgroup of $T^k$ generated by the corresponding isotropy subgroups of each opposing facet, that is, we consider the diagonal circle in the corresponding $T^2$ generated by the isotropies of each facet that intersects each of these isotropies trivially. Since the facets are opposing, they will not intersect in any lower-dimensional face and hence the diagonal circle in this subgroup 
will intersect all isotropy subgroups trivially and corresponds to a circle subgroup of $T^k$ that acts freely.

Recall that  $\Sigma^{n_i}$ has $n_i$ facets and $\Delta^{n_i}$ has $n_i + 1$ facets and  that the number of facets in $P^{n-k}=\prod_{i<r} \Sigma^{n_i} \times \prod_{i\geq r}^s \Delta^{n_i}$ 
equals the sum of the facets in each  of the $\Sigma^{n_i}$ and  $\Delta^{n_i}$. That is, if we let $f(P^{n-k})$ denote the total number of facets, then $$f(P^{n-k})=\sum_{i<r} (n_i)+\sum_{i\geq r}^s (n_i+1)=\dim(P^{n-k}) +s-r+1,$$ where $s-r+1\geq 0$ is equal to the total number of all simplices in $P^{n-k}$.  
Now, since $P^{n-k}$ has $k$ facets by assumption, this means that $P^{n-k}$ contains the 
product of $2k-n$ simplices.

It is clear that $\Sigma^{n_i}$, since it is the lunar suspension of a $\Delta^{n_i-1}$, has no pairs of opposing facets that do not intersect in a lower dimensional face. 
Thus, in order to find the $(2k-n)$ pairs of opposing facets, it suffices to show that for any product of $(2k-n)$ of the $\Delta^{n_i}$, that we can find $(2k-n)$ opposing facets. The product of these pairs of opposing facets with the remaining product of $\Sigma^{n_i}$s in $P^{n-k}$ will then form 
the desired set of $(2k-n)$ pairs of opposing facets.

 We will use barycentric coordinates for the $n$-simplex, $\Delta^n = \sum_{i=0}^{n} s_i \,v_i$ with vertices $v_0,\dots,v_n$ where $(s_0,\dots s_n) \in \rrr^{n+1},\,\sum_{i=0}^{n} s_i = 1$ and $s_i \ge 0$ for $i=0,\dots n$.   We define  the set $$\Delta^{n-1}_{i} = \{(s_0,\dots,s_{i-1},0,s_i,\dots s_n)| \sum_{i=0}^{n} s_i = 1,\, s_i \ge 0 \,\, \text{for} \,\,i=0,\dots n\}$$ corresponding to the $(n-1)$-simplex to be {\it opposite} the vertex $$v_i = (0,\dots 0,1,0,\dots 0)$$ with $1$ in the $i$-th coordinate.  

We will denote an opposing pair of faces  in $P^n$ by  $(F_1, F_2)$, where we allow the dimension of the $F_i, i=1, 2$, to vary from $0$ to $n-1$.
 Consider the following two canonical examples.
 
 \begin{example} \label{ex1}
Consider the simplex, $\Delta^{k}$, with $k+1$ facets,  arising as the quotient of a $T^{k+1}$-action on $S^{2k+1}$.  
Consider the pair of opposing facets given by  
\bdm (\Delta^{k-1}_{i}, v_i).
\edm
Let $T_i^1$ be the isotropy subgroup corresponding to  $\Delta^{k-1}_{i}$  and let $T_i^{k}$ be the isotropy subgroup corresponding to $v_i$.
Then it is clear that $T_i^1$ and $T_i^{k}$ intersect trivially in $T^{k+1}$. The diagonal circle in $T^{k+1}$, which is generated by  $T_i^1$ and  $T_i^{k}$, intersects all isotropy groups trivially and 
hence acts freely.

\end{example}

 \begin{example}\label{ex2}

Consider the simplex, $\Delta^{k-l}\times \Delta^l$, with $k+2$ facets,  arising as the quotient of a $T^{k+2}$-action on $S^{2k-2l+1}\times S^{2l+1}$.   
Consider the two pairs of opposing facets given by  
\bdm (\Delta^{k-l-1}_{i}\times \Delta^l, v_i\times \Delta^l) \textrm{ and  } (\Delta^{k-l}\times \Delta^{l-1}_j,  \Delta^{k-l}\times v_j).
\edm
Let $T_i^1$ be the isotropy subgroup corresponding to  $\Delta^{k-l-1}_{i}\times \Delta^l$ and let $T_i^{k-l}$ be the isotropy subgroup corresponding to $v_i\times \Delta^l$ and $T_j^1$ be the isotropy subgroup corresponding to  $\Delta^{k-l}\times \Delta^{l-1}_j$ and let $T_j^{l}$ be the isotropy subgroup corresponding to $\Delta^{k-l}\times v_j$.
Then it is clear that $T_i^1$ and $T_i^{k-l}$ intersect trivially in $T^{k+2}$ and so do $T_j^{1}$ and $T_j^{l}$. The diagonal circles in $T^{k-l+1}=\langle T_i^1, T_i^{k-l}\rangle$ and in  $T^{l+1}=\langle T_j^1, T_j^{l}\rangle$ intersect all isotropy groups trivially and 
hence act freely.

\end{example}

For the sake of simplicity of notation, we will prove only the case when $P^{n-k}= \prod_{i=1}^{2k-n} \Delta^{n_i}$.
We will proceed by induction on the number of simplices contained in the product.  The base case is covered by Example \ref{ex1}.  
Let $ \prod^l_{\Delta} = \prod_{i=1}^l \Delta^{n_i}$ with $\sum_{i=1}^l n_i +l=k$ facets arising as the quotient of a $T^k$-action on $M^n$,
with free rank equal to $2k-n=l$, and assume by the induction hypothesis that on any subproduct $ \prod^{l-1}_{\Delta}\subset \prod^l_{\Delta}$ we can find $(l-1)$ pairs of opposing facets. 

Without loss of generality, we let $ \prod^{l}_{\Delta} =  \prod^{l-1}_{\Delta} \times \Delta^{n_{l}}$.
We denote $\Delta^{n_1} \times \dots \times \widehat{ \Delta^{n_j}} \times  \Delta^{n_l}$ by $ \prod^{l,j}_{\Delta}$.   
Then we get $l$ freely acting circles from opposing facets given by $\{ \prod^{l,j}_{\Delta}    \times \Delta^{n_{j}-1}_i, \prod^{l,j}_{\Delta} \times v_i\} $
 for $j=1, \dots, l-1$  and one additional one by the construction in Example \ref{ex2} with the pair of opposing facets given by 
$\{\prod^{l-1}_{\Delta} \times \Delta^{n_{l}-1}_i,  \prod^{l-1}_{\Delta} \times v_i\}\,.$  Hence we get $l$ pairs of opposing facets, as desired.
\end{proof}

Recall that given an $n$-dimensional homology polytope $P$, we associate to it  the moment angle manifold, $\mathcal{Z}_{P}$  (see Definition \ref{moment}). Note that since the $\Sigma^{n_i}$ are lunar suspensions of simplices $\Delta^{n_i-1}$, we may define the barycenter of any $\Sigma^{n_i}$ to be  the corresponding barycenter of the $\Delta^{n_i-1}$ over which we suspend to obtain $\Sigma^{n_i}$. The {\em barycenter of $P$} is then defined to be the
 $n$-tuple of the barycenters of each  $\Sigma^{n_i}$ and $\Delta^{n_j}$ in the product.
It is then clear that  a homology polytope of the form $P=\prod_{i<r} \Sigma^{n_i} \times \prod_{i\geq r} \Delta^{n_i}$ admits a conical decomposition by taking the barycenter of $P$ to be the cone point for each of the cones over the facets. Further, by Remark \ref{ZPlstd}, we see that the 
 $T^k$ action on $\mathcal{Z}_{P}$ is locally standard.  
 Hence, 
the following lemma now follows by a direct application of Theorem \ref{t:ecst}.

\begin{lemma}\label{Exist} Let $\phi$ be the weight preserving diffeomorphism of Lemma \ref{Smith} and $\mathcal{Z}_{\phi(P^{n-k})}$ the moment angle manifold corresponding to $\phi(P^{n-k})$.  
The $T^k$ action on $\mathcal{Z}_{\phi(P^{n-k})}$  admits a cross-section.
\end{lemma}

The existence of a  locally standard action implies that  all singular isotropies are connected and correspond to points on the boundary and all interior points correspond to principal orbits. So we may apply  Lemmas \ref{Smith},  \ref{Exist}, and \ref{Free}, in combination with the Cross-Sectioning Theorem \ref{t:ecst} and the Equivariant Classification Theorem \ref{t:ect} to obtain the following theorem.

\begin{theorem}\label{c:wpd} 
Let  $T^k$ act isometrically, effectively on  a closed, simply connected, $n$-dimensional Riemannian manifold, 
whose quotient space $M/T$ is   
\begin{equation*}
P^{n-k}=\prod_{i<r} \Sigma^{n_i} \times \prod_{i\geq r} \Delta^{n_i},
\end{equation*}
and the number of facets of $P^{n-k}$ is equal to $k$. Suppose further that the $T^k$ action is locally standard. 

Then there exists an equivariant diffeomorphism between $M$ and $\bar{M}$, where $\bar{M}=\mathcal{Z}_{\phi(P^{n-k})}$ admits a freely acting torus of rank $2k-n$ and $\phi$ is the weight preserving diffeomorphism of Lemma \ref{Smith}.
\end{theorem}


\section{Almost non-negative curvature and locally standard actions}\label{s4}

In this section, using a generalization of work of  \cite{Spi}, we extend a result of \cite{Wie} that allows us to determine 
 when a torus manifold with non-negatively curved quotient space is locally standard.

In fact, one can see from the proof of  Theorem \ref{Spindeler},  that in order to obtain the disk bundle decomposition,  we only need assume that $M$ is closed and the quotient $M/G$ is a non-negatively curved Alexandrov space. That is,  we can reformulate Theorem \ref{Spindeler}, as follows.

\begin{theorem}\label{newSpindeler}  Assume that $G$ acts fixed point homogeneously  on a closed  Riemannian manifold $M$ such that $M/G$ is a non-negatively curved Alexandrov space. Let $F$ be a fixed point component of maximal dimension. Then there exists a smooth submanifold $N$ of $M$, without boundary, such that $M$ is diffeomorphic to the normal disk bundles $D(F)$ and $D(N)$ of $F$ and $N$ glued together along their common boundaries;
\bdm
M  = D(F) \cup_{\partial} D(N).
\edm
Further, $N$ is $G$-invariant and all points of $M\setminus \{F\cup N\}$ belong to principal $G$-orbits. 

\end{theorem}

One can also see from the proof of Proposition \ref{Spindelerprop},  that its hypotheses can be weakened   to assume only that $M$ is closed and the quotient $M/G$  is a non-negatively curved Alexandrov space as follows.
  
  \begin{proposition}  \label{newSpindelerTorus}  Assume that $M$ is a closed, simply connected 
torus manifold and let $M, N$ and $F$ be as in Theorem \ref{newSpindeler}. Then $N$ has codimension greater than or equal to $2$ and $F$ is simply connected.
  \end{proposition}

Let $G$ act on $M^n$, a closed Riemannian manifold.  Then the following lemma shows that $G$-invariant components of the fixed point set of any subgroup $H\subset G$ descend to Alexandrov spaces with the same lower curvature bound as the orbit space $M/G$. In particular, this result gives a refinement of work of Grove, Moreno, and Petersen \cite{GMP}, showing that $\partial(M/G)$ is an  Alexandrov space with the same lower curvature bound as the orbit space $M/G$.

 \begin{lemma}\label{condition} Let $M$ be a closed Riemannian manifold with an isometric $G$-action.  For any  $G$-invariant component $F\subset M^{H}$, $H\subset G$, the quotient space, $F/G$, is an Alexandrov space with the same lower curvature bound as $M/G$.  
\end{lemma}
\begin{proof} 
Since M is a closed Riemannian manifold, it has a positive injectivity radius, which we denote by $\epsilon >0$.  The fixed point set component $F$ is a closed, totally geodesic submanifold. Being totally geodesic, it is $\epsilon$-convex, that is, any geodesic of length less than $\epsilon$ between points $p, q\in F$ lies entirely in $F$ (see Gromoll and Grove \cite{GG}). Moreover, by the assumption on the injectivity radius, this geodesic will be unique. Since $F$ is a closed submanifold of $M$, it follows that $F/G$ is a compact Alexandrov space, possibly with boundary.

Now let $\bar{p} \,,\,\bar{q} \in F/G$  such that $\bar{q}\in B_{\epsilon}(\bar{p})$. Let $\bar{\gamma}$ be a shortest geodesic between $\bar{p}$ and $\bar{q}$ in $M/G$. Then $\bar{\gamma}$ lifts to a horizontal geodesic $\gamma$, between $p$ and $q$, the inverse images of $\bar{p}$ and $\bar{q}$. 
But since $\bar{\gamma}$ and hence $\gamma$ are of length less than $\epsilon$, $\gamma$ is entirely contained in $F$. But this means that $\bar{\gamma}$ is entirely contained in $F/G$.
Thus $F/G$ is $\epsilon$-convex in $M/G$. Since a subset, $Y$, of an Alexandrov space, $X$, that is itself an Alexandrov space  and is $\epsilon$-convex will satisfy  local triangle comparison, it follows that $Y$ will have the same lower curvature bound as $X$. Therefore  $F/G$ has the same lower curvature bound as $M/G$.
\end{proof}

Since fixed point set components of subtori are invariant under the full torus action, the proof of the following generalization of  Lemma 6.3 of \cite{Wie}, which we leave to the reader, now follows by Lemma \ref{condition} and the same arguments as in the proof of Lemma 6.3 in \cite{Wie}, using induction on the dimension of $M$.

\begin{theorem}\label{t:locallystandard} Let $M$ be a closed, simply connected, torus manifold  and assume that $M/T$ is a non-negatively curved Alexandrov space. Then $M$ is locally standard and $M/T$ and all its 
faces are diffeomorphic (after smoothing the corners) to standard disks. In particular, $H^{\text{odd}}(M)=0$.
\end{theorem}

One can also generalize Theorem \ref{t:torus} for the class of manifolds considered in Theorem \ref{t:locallystandard}.  The only element in the proof of Theorem \ref{t:torus}, that we have not already generalized to this class of manifolds and which is curvature dependent is Lemma 4.2 of \cite{Wie}. 
However,  it is clear that Lemma 4.2 of \cite{Wie} does hold, since we require the quotient space to be non-negatively curved for this class of manifolds, and thus, one can still apply Lemma 4.1 in \cite{GGS2} to obtain the result. Thus, combining 
our Theorems \ref{newSpindeler} and \ref{t:locallystandard},  along with Lemmas 6.4, 6.7, 6.8 in \cite{Wie} and an argument from the proof of Theorem 4.1 in \cite{Wie}, yields the following theorem.

\begin{theorem}\label{generaltorus} Let $M$ be a closed, simply connected, torus manifold and assume that $M/T$ is a non-negatively curved Alexandrov space. Then $M$ is equivariantly diffeomorphic to the quotient of a product of spheres of dimensions greater than or equal to three by a free, linear torus action.
\end{theorem}

  \begin{remark} Using Theorem B of Searle and Wilhelm \cite{SW}, which allows one to lift a metric of almost non-negative curvature to a $G$-manifold $M$, provided $M/G$ is almost non-negatively curved,  we note that $M$ as above admits a $G$-invariant family of metrics of almost non-negative curvature. So Theorem \ref{newSpindeler} and Proposition \ref{newSpindelerTorus},  hold for the class of $G$-invariant almost non-negatively curved manifolds with non-negatively curved quotient spaces. Moreover,  Theorems \ref{t:locallystandard} and \ref{generaltorus} also hold 
 for the class of $G$-invariant almost non-negatively curved manifolds with non-negatively curved quotient spaces.
  
  \end{remark}


\section{A General lower bound for the free rank}\label{generallowerbound} In order to establish the lower bound for the free rank, we first need the following proposition, which establishes the existence of a $T^k$ fixed point when the torus action has no circle subgroup acting almost freely.

\begin{proposition}\label{l:fixed_point} Let $T^k$ act isometrically  on $X^n$, a closed $n$-dimensional Alexandrov space with a lower curvature bound. Suppose that no circle subgroup acts almost freely, or equivalently that every element $t\in T^k$ has a fixed point. Then $T^k$ has a fixed point.
\end{proposition}

The proof uses the same ideas as  the proof of Lemma 5.5 of Harvey and Searle \cite{HS}. However, since it is short, we include it here  for the sake of completeness.
\begin{proof}
Consider a dense $1$-parameter subgroup of $T^k$, and within it an infinite cyclic
subgroup. By assumption, every element $t\in T^k$ has a fixed point. 
Thus the cyclic subgroup fixes a point. As we move the generator
of the cyclic subgroup towards the identity, we generate a sequence of fixed points in
$X$, and any limit point of that sequence will be fixed by the torus.
\end{proof}

The following proposition establishes a lower bound for the free rank of a general isometric torus action on an Alexandrov space with an arbitrary lower curvature bound.

\begin{proposition}\label{p:free rank} Let $T^{k}$ act isometrically and effectively on $X^{n}$, a closed Alexandrov space  with a lower curvature bound, and $k \geq \lfloor (n+1)/2\rfloor$.
Then the free rank of the $T^{k}$-action is greater than or equal to $2k - n$.
\end{proposition}

\begin{proof}
 Let $T^l \subset T^{k}$ be the largest subtorus that acts almost freely, and suppose that  $l<  2k - n$. Since 
 $l < 2k - n$, it follows that $$k-l >  \frac{n-l}{2} \geq  \lfloor \frac{n-l}{2}\rfloor.$$ 

Now  $T^{k-l}\cong T^k/T^l$ acts on $X^{n-l} = X^{n}/T^{l}$, a closed Alexandrov space with the same lower curvature bound, and no circle subgroup of $T^{k-l}$ acts almost freely on $X^{n-l}$.   By Lemma \ref{l:fixed_point}, there is a point $\bar{p}\in X^{n-l}$ fixed by $T^{k-l}$.
 So, 
there is an action of $T^{k-l}$ on $\Sigma_{\bar{p}}$, the unit normal space of directions to this orbit,  which is itself a closed Alexandrov space of dimension $n-l-1$ with curvature bounded below by $1$.  
The  Maximal Symmetry Rank Theorem  for positively curved Alexandrov spaces  in \cite{HS} states that for a rank $j$ torus-action on $X^m$,  as is also true in the manifold case, $j\leq \lfloor (m+1)/2\rfloor$.  
We then have
$$k-l\leq \lfloor  \frac{n-l}{2}\rfloor,$$
 a contradiction. 
Hence the bound holds.
\end{proof}

Combining Propositions  \ref{l:fixed_point} and \ref{p:free rank}  with Corollary II.6.3 of  \cite{Br} yields the following corollary.
\begin{corollary}\label{l:torus} Let $T^k$ act isometrically and effectively on $M^n$, a closed, simply connected, non-negatively curved Riemannian manifold, with $k \geq \lfloor (n+1)/2\rfloor$.
Suppose that the free rank of the action is less than or
equal to  $2k-n$. Then, the following hold:
\begin{enumerate}
\item If the free dimension  is equal to the free rank, then the quotient space, $M^{2n-2k}=M^n/T^{2k-n}$, admits a $T^{n-k}$-action and is a closed, simply connected, 
torus manifold of non-negative sectional curvature.
\item If the free dimension  is not equal to the free rank, then the quotient space, $X^{2n-2k}=M^n/T^{2k-n}$, admits a $T^{n-k}$-action and is a closed, simply connected, non-negatively curved (in the Alexandrov sense) torus orbifold.
\end{enumerate}
\end{corollary}


\section{Proof of Theorem \ref{t:thma}}\label{s:6}

 Recall that in Theorem \ref{t:thma} we have an isometric and effective $T^k$ action on $M^n$, a closed, simply connected, non-negatively curved Riemannian manifold. Moreover, we assume that the $T^k$ action is isotropy maximal, that is, the free rank  is equal to $2k-n$. We  show in Proposition \ref{lstd} that in the presence of non-negative curvature this implies the torus action is locally standard, which is key in the proof of Theorem \ref{t:thma}.   
  
To prove the equivariant classification we split the proof into two cases: Case (1), where the free dimension equals the free rank,  and Case (2), where the subtorus corresponding to the free rank acts almost freely, but not freely. In both cases 
there are two further subcases: Subcase (a), where the rank of the action is equal to the number of facets of the orbit space $M^{n}/T^{k}$ and Subcase (b), where the rank of the action is strictly less than the number of facets of the orbit space $M^{n}/T^{k}$.

\subsection{Isotropy Maximal Torus Actions}
Before we begin with the proof of Theorem \ref{t:thma}, we show that isotropy-maximal torus actions behave nicely in the presence of non-negative sectional curvature.
\begin{proposition}\label{lstd} Let $T^k$ act isometrically, effectively and isotropy-maximally on $M^{n}$, a closed,  simply connected, non-negatively curved Riemannian manifold. Then the following are true:
\begin{enumerate}
\item The torus action on $M$ is locally standard,
in particular,  $M/T$ is a nice manifold with corners, such that the isotropy groups are constant on all open faces of $M/T$; and
\item All closed faces of $M/T$ are diffeomorphic to standard discs, after smoothing the corners.
\end{enumerate}
\end{proposition}
\begin{proof}
The proof follows along the same lines as the proof of Lemma 6.3 in \cite{Wie} and is 
 by induction on the dimension of the orbit space.
For simplicity of notation, let $T$ denote the torus $T^k$ throughout.
Let $T(x)$ be a minimal orbit in $M$. Then there is a codimension-two submanifold
$F\subset M$ fixed by a circle subgroup $C$ of $T$ containing $T(x)$. Note that the action of $T/C$ on $F$ is isotropy-maximal.
By Theorem \ref{Spindeler}, $M$ admits  an equivariant disk bundle decomposition as  $$M =D(F)\cup_E D(N),$$
and by Lemma 3.29 of \cite{Spi}, we have that $\codim(N) \geq 2$.  

We split the proof into two cases: Case ($1$), where  $\codim(N) \geq 3$, and Case ($2$), where $\codim(N)= 2$.
We treat Case ($1$) first.

In the case where $\codim(N) \geq 3$, as in the proof of Part (2) of Proposition 6.2 of \cite{GGS2}, we see that $F$ must be simply-connected as follows. Let $\gamma$ be a loop in $F$. Since $M$ is simply connected, $\gamma$ bounds a $2$-disk, $D^2$.
Since $\codim(N) \geq 3$, by transversality, we can perturb $D^2$ so as to lie in the complement of $D(N)$, while keeping $\gamma=\partial D^2$ in $F$.  Since $D^2$ lies in $D(F)$, it deformation retracts onto $F$ and the conclusion follows.

We can now apply the induction hypothesis to $F$ to get that the action of $T/C$ on $F$ is locally standard and the faces of $F/T$ are standard discs.
Observe that $E$ is a $C$-bundle over $F$. Let $x'\in E$ such that $\pi(x')=x$. The orbit $T(x')$ in $E$  is an orbit of type $T/T'$ where $T'$ is some codimension-one subtorus of $T_x$ such that $T_x = C \times T'$.
Since the $T'$-action on $F$ is locally standard and $T(x)$ is a minimal orbit, there is a neighborhood of $T(x')$ which is equivariantly diffeomorphic to
$T/T' \times W \times \rrr$,
where $W$ is a faithful $T'$-representation of dimension $2 \dim(T')$ and $T$ acts trivially on $\rrr$. Note that this $\rrr$-factor is normal to $E$.
Let $T(y)$ be the projection of the orbit $T(x')$ to $N$. Since $\dim(T(x')) = \dim(T(x)) + 1$ there are two cases:
\begin{enumerate}
\item[ ($1.a$)] $T(y)$ is a minimal orbit, that is, $\dim(T(y)) = \dim(T(x))$; or
\item[($1.b$)] $T(y)$ is not a minimal orbit, that is, $\dim(T(y))=\dim(T(x))+1$.
\end{enumerate}
Assume first that we are in Case ($1.a$). Then $T(y)$ is a minimal orbit. Moreover a neighborhood of $T(y)$ is equivariantly diffeomorphic to
$T/T_y \times W \times \ccc$
where $T_y/T'$ acts on $\ccc$ by rotation and $\ccc$ is normal to $N$. Since $N$ is an invariant submanifold, it follows that $N$ is a fixed point component of some subtorus $T''$ such that $T_y\subset T''\subset T$ . Since $T(y)$ is minimal it follows that $N$ is contained in some codimension-two submanifold $F'$ fixed by a circle subgroup of $T$. Applying the above arguments for $F'$ instead of $F$, we see that the action on $F'$ is locally standard and that each face of $F'$ is diffeomorphic to a standard disc after smoothing corners. In particular, this is also true for the action on $N$. The rest of this case is as in the first case of the proof of  Lemma 6.3 \cite{Wie}.
In Case ($1.b$) $T'$ is the identity component of $T_y$. 
As in the proof of Lemma 6.3 \cite{Wie}, it follows that $S^1$ acts freely on $N$.
Moreover $N$ is fixed by a torus $T''\subset T'$ with $\codim(N) = 2\dim(T'') + 1$. Hence it follows that $N/C$ is a closed, non-negatively curved, simply connected manifold with an isotropy-maximal torus action, and so the induction hypothesis applies to $N/C$. The rest of the proof is then as in  \cite{Wie}.

We now treat Case ($2$). Note that in this case we only know that $\pi_1(F)$ is cyclic (see Theorem 3.1 of \cite{ES1}). Therefore there are two cases:
\begin{enumerate}
\item[(2.a)] There is an $S^1\subset T^k$ which acts almost freely on $F$; or
\item[(2.b)] $F$ has finite cyclic fundamental group.
\end{enumerate}
In Case ($2.b$),  the argument for Case ($1$) can be applied 
to the universal cover of $F$ with the lifted torus action. Theorem I.9.1 of \cite{Br} shows that $\tilde{F}/T=F/T$ and that the isotropy subgroups will be the same, since the cover has a finite number of sheets. The rest of the argument then proceeds as in Case ($1$).

If the $S^1$-action is free in Case ($2.a$), one can apply the induction hypothesis to $F/S^1$, which is a closed, simply connected, non-negatively curved manifold with an isotropy-maximal action.
If the $S^1$-action is almost free, then we lift the action to $\tilde{F}$, the universal cover of $F$,  which splits metrically as the product of $\bar{F}^{n-3}\times \rrr$, where $\bar{F}^{n-3}$ is a closed, non-negatively curved, simply connected manifold by the Splitting Theorem of Cheeger and Gromoll \cite{CG}.
In this case, by Theorem I.9.1 of \cite{Br}, we see that the $S^1$ action lifts to an action of $\rrr$ on $\tilde{F}=\bar{F}^{n-3}\times \rrr$. By work of Hano \cite{Hano}, the isometry groups of $\tilde{F}$ split into the product of the isometry groups of the factors, which tells us that the $\rrr$ action is trivial on $\bar{F}^{n-3}$ and free on  $\rrr$. Since $\bar{F}^{n-3}=\tilde{F}/\rrr=F/S^1$, we see that we may once again apply the induction hypothesis to $F/S^1$.

\end{proof}


\subsection{Proof of Case (1) of Theorem \ref{t:thma}}\label{ss51}

  For the sake of simplicity and consistency of notation in what follows,
we set the rank $k$ of the torus action to be $n+p$ and the dimension $n$ of the manifold to be $2n+p$, that is, we set
$$\begin{cases}
k\mapsto n+p\\
n\mapsto 2n+p.\\
\end{cases}$$
Thus, we consider a $T^{n+p}$-action on $M^{2n+p}$ with free rank equal to $p$ and $\text{dim}(M/T)=n$.
Recall that we assume here that the free dimension of the torus action is equal to the free rank.  
As detailed above, we now break the proof into two cases.


\subsubsection{\bf Case (${\bf 1.a}$): The number of facets of ${\bf M/T}$ is equal to the rank of the torus action}
We first consider the following more general situation: let $M^{2n+p}$ be a closed, simply connected,  principal $T^{p}$-bundle over $N^{2n}$, where $N^{2n}$ has a locally standard, smooth $T^n$-action with orbit space $P^n$ as in Display \eqref{2.2}
and the number of facets of  $P$  is equal to $n+p$. In the following theorem we show that such an $M$ 
is  equivariantly diffeomorphic to the moment angle manifold $\mathcal{Z}_P$ as in Display \eqref{2.3}.

 \begin{theorem}\label{momentangle}  Let $M^{2n+p}$ be a closed, simply connected,  principal $T^{p}$-bundle over $N^{2n}$, where $N^{2n}$ has a locally standard, smooth $T^n$-action with orbit space $P^n$ as in Display \eqref{2.2}
and  $n+p$ facets.
 Then $M^{2n+p}$ is equivariantly diffeomorphic to the moment angle manifold $\mathcal{Z}_P$. 
 \end{theorem}
  \begin{proof}

 Since $M$ is a principal $T^p$-bundle over $N$, it is clear that the $T^{n+p}$ action on $M$ is locally standard.
 Then Theorem \ref{c:wpd} establishes that $M$ is equivariantly diffeomorphic to $\mathcal{Z}_{\phi(P)}=\mathcal{Z}_{P}$, the moment angle manifold corresponding to $\phi(P)$, where $\phi:P\rightarrow \phi(P)$ is the weight preserving diffeomorphism of Lemma \ref{Smith}, as desired.
 The diagram below illustrates this case.
 \vspace{.5cm}
\begin{equation*}
\begin{split}
\xymatrix{
& {} &  M^{2n+p} \ar@/_2pc/[dd]_{/T^{n+p}}  \ar[r]^{\cong}   \ar[d]^{/T^{p}}  &  \mathcal{Z}_P^{2n+p}    \ar[ld]^{/T^{p}}   &    \\
& {} & N^{2n}    \ar[d]^{/T^n}    \\
& {}  & P^n    & {}    &  {} }
\end{split}
\end{equation*}
\end{proof}

 We now observe that if  $M^{2n+p}$ is a closed, simply connected, non-negatively curved  manifold, admitting an isotropy maximal  
$T^{n+p}$-action then
Part (1) of Corollary \ref{l:torus} tells us that  $N^{2n}=M^{2n+p}/T^{p}$, the quotient of $M^{2n+p}$ by the free $T^{p}$ action, is a non-negatively curved torus manifold.
 Lemma \ref{locallystandard} gives us that the torus action on $N$ is locally standard.
Proposition \ref{Q} tells us that the orbit space of $N$, $P=M^{2n+p}/T^{n+p}=N^{2n}/T^n$, is a product of simplices and lunar suspensions, as in Display \eqref{2.2}.

The following corollary of Theorem \ref{momentangle} is then immediate, thus proving Case ($1.a$) of Theorem \ref{t:thma}.
  \begin{corollary}\label{momentangle''}  Let $M^{2n+p}$ be a closed, simply connected,  non-negatively curved manifold admitting an isotropy maximal $T^{n+p}$ action with free rank equal to the free dimension and such that the   
  orbit space $P^n$ has $n+p$ facets. Then $P^n$ is as in Display \eqref{2.2} and 
  $M^{2n+p}$ is equivariantly diffeomorphic to the moment angle manifold $\mathcal{Z}_P$. 
 \end{corollary}
 

\subsubsection{\bf Case ($1.b$): The number of facets is strictly greater than the rank of the torus action}
 Again, we consider the following more general situation: let  $M^{2n+p}$ be a closed, simply connected,  principal $T^{p}$-bundle over $N^{2n}$, where $N^{2n}$ has a locally standard, smooth $T^n$-action with orbit space $P^n$ as in Display \eqref{2.2}
and the number of facets of  $P^n$  is strictly greater than $n+p$, with $p>0$.
The following theorem establishes that such an $M$ 
 is  equivariantly diffeomorphic to the quotient of the moment angle manifold $\mathcal{Z}_P$ by a free, linear torus action.

\begin{theorem}\label{t:generalmomentangle}  Let $M^{2n+p}$ be a closed, simply connected  principal $T^{p}$-bundle over $N^{2n}$ where $N^{2n}$ has a locally standard, smooth $T^n$-action with orbit space $P^n$ as in Display \eqref{2.2}, with $m$ facets, where $m > n+p$. We further assume that the bundle of principal orbits is trivial.
Then there is a smooth, principal $T^{m-n-p}$-bundle $\pi: Y^{n+m} \longrightarrow M^{2n+p}$, with $Y^{n+m}$ a closed, simply connected manifold that is $T^m$-equivariantly diffeomorphic to the moment angle manifold $Z_P$.  Thus $M^{2n+p}$ is equivariantly diffeomorphic to the quotient of $\mathcal{Z}_P$  by a $T^{m-n-p}$-action.
\end{theorem}

\begin{proof}
Since $M$ is a principal $T^{p}$-bundle over $N^{2n}$, and $H^2(N^{2n}) \cong \zzz^{m-n}$ (see, for example, Theorem 7.4.35 in \cite{BP1}),  using the Leray-Serre spectral sequence or the homotopy sequence for the bundle we obtain that $H^2(M; \zzz)$ contains $\zzz^{m-n-p}$ as a subgroup. 
Recall that if $Y$ is a principal $T^{m-n-p}$-bundle over $M$ then it is classified by homotopy classes of maps of $M$ into $BT^{m-n-p}$, hence by $[M, BT^{m-n-p}]$.  Since
\bdm
[M, BT^{m-n-p}] \cong [M, BS^1] \times ... \times [M, BS^1] \cong \oplus_{i=1}^{m-n-p} H^2(M; \zzz),
\edm
and $H^2(M; \zzz)$ is non-trivial, it follows that there is a non-trivial principal $T^{m-n-p}$-bundle over $M$, $Y^{m+n}$.

 We now use the fact that $Y$ is a $T^{m-n}$-principal bundle over $N^{2n}$, since in the category of finite dimensional manifolds the composition of two principal bundles is again a principal bundle (see, for example, McKay \cite{M}).  

We must now show that $Y^{m+n}$ can be chosen to be simply connected.
To do so, we apply the argument used in the proof of Proposition 6.4 in \cite{D} to show that that $H_1(Y)$ is trivial. Since $Y$ is a principal $T^{m-n-p}$ bundle over $M$ and $M$ is simply connected, it follows from the long exact sequence in homotopy associated to the fibration that $\pi_1(Y)$ is abelian. Since $H_1(Y)=0$ this implies that $\pi_1(Y)=0$.

 Since $Y$ is simply connected, this means that we are now in the setting of Theorem \ref{momentangle} .  Hence $Y^{m+n}$ is equivariantly diffeomorphic to the moment angle manifold $\mathcal{Z}_P$ as claimed.   The diagram below illustrates this case.

\vspace{.5cm}

\begin{equation*}
\begin{split}
\xymatrix{
 & Y^{n+m}  \ar[rrr]^{\cong}  \ar[d]^{/T^{m-n-p}}& {} & {}  & \mathcal{Z}_P^{n+m}   \ar[llld]^{/T^{m-n-p}}    \ar@/^1pc/[llldd]^{/T^{m-n}}   &   \\
 &  M^{2n+p} \ar@/_2pc/[dd]_{/T^{n+p}}       \ar[d]^{/T^p}  &  {}    &  {}    \\
 &  N^{2n} \ar[d]^{/T^n}   \\
 &  P^n    & {}    &  {} }
\end{split}
\end{equation*}
\end{proof}
As in Case ($1.a$),  given $M^{2n+p}$ a closed, simply connected,  non-negatively curved manifold admitting an isotropy maximal $T^{n+p}$ action with free rank equal to the free dimension and such that the orbit space $P^n$ has $m$ facets where $m > n+p$, it follows that the quotient of $M$ by the free $T^{2k-n}$ action is a non-negatively curved torus manifold with a locally standard  torus action  whose orbit space $P$ is 
 as in Display \eqref{2.2}. The following corollary of Theorem \ref{t:generalmomentangle} is then immediate, thus proving Case ($1.b$) of Theorem \ref{t:thma}.

 \begin{corollary}\label{generalmomentangle} Let $M^{2n+p}$ be a closed, simply connected,  non-negatively curved manifold admitting an isotropy maximal $T^{n+p}$ action with free rank equal to the free dimension and such that the  the orbit space $P^n$ has $m$ facets where $m > n+p$.
Then there is a smooth principal $T^{m-n-p}$-bundle $\pi: Y^{n+m} \longrightarrow M^{2n+p}$. Further, $Y^{n+m}$ is simply connected and is $T^m$-equivariantly diffeomorphic to the moment angle manifold $\mathcal{Z}_P$.  
 Thus $M^{2n+p}$ is equivariantly diffeomorphic to the quotient of $\mathcal{Z}_P$  by a free, linear $T^{m-n-p}$-action.
\end{corollary}


\subsection{The proof  of Case ($2$) of Theorem \ref{t:thma}}\label{ss52}

In this section we will prove Case ($2$) of Theorem \ref{t:thma}, namely when the free dimension does not equal the free rank. 
Again, for the sake of simplicity and consistency of notation in what follows,
we set the rank $k$ of the torus action to be $n+p$ and the dimension $n$ of the manifold to be $2n+p$, that is, we set
$$\begin{cases}
k\mapsto n+p\\
n\mapsto 2n+p.\\
\end{cases}$$
Thus, we consider a $T^{n+p}$-action on $M^{2n+p}$ with free rank equal to $p$ and $\text{dim}(M/T)=n$.
Recall that by Corollary \ref{l:torus}, $M^{2n+p}/T^{p}=X^{2n}$ is a non-negatively curved torus orbifold admitting a $T^{n}$ isometric action.   

Applying Theorem \ref{misre}, we see that $M^{2n+p}$ is rationally elliptic.  For quotients of rationally elliptic manifolds admitting almost free torus actions, we have the following observation.

\begin{observe}\label{MTisre} Let $M$ be a closed,  rationally elliptic manifold admitting an effective, isometric almost free torus action. Then $M/T$ is a closed, rationally elliptic orbifold.
\end{observe}
\begin{proof}
The simple connectivity of $M/T$ follows directly from Corollary II.6.3 of \cite{Br}.  In addition, since $M/T$ is closed, the  finiteness condition on the rational cohomology groups is automatically satisfied.

Moreover, since  $M$ is rationally elliptic, the  long exact sequence in homotopy associated to the fibration $$ T \hookrightarrow M \rightarrow M\times_T ET$$ tells us that the finiteness condition on the rational homotopy groups of $M\times_T ET$ is satisfied.
  The  long exact sequence in homotopy associated to the fibration 
 $$ET\hookrightarrow M\times_T ET\rightarrow M/T\,,$$
 obtained by projection onto the first factor, 
 implies that  $\pi_i(M\times_E ET)\cong\pi_i(M/T)$ for all $i\geq 1$.   Hence $M/T$ is rationally elliptic.
 \end{proof}

By Observation \ref{MTisre},
$M^{2n+p}/T^{p}=X^{2n}$ is a simply connected,  non-negatively curved, rationally elliptic torus orbifold. The proof of Theorem \ref{REtorusorbifold} shows that the face poset of $M/T$ is combinatorially equivalent to that of the face poset of $P$, where $P$ is as in Display \eqref{2.2}. Theorem 4.2 of \cite{D} then gives us that $M/T$ is as in   Display \eqref{2.2}.


\subsubsection{\bf Case ($2.a$): The number of facets of ${\bf M/T}$ is equal to the rank of the torus action}

We have just seen that the orbit space, $M/T=P$, is as in Display \eqref{2.2} and 
by Proposition \ref{lstd}, the torus action on $M$ is locally standard.   It then follows by the Equivariant Cross-Sectioning Theorem \ref{t:ecst} that a cross-section for the action on $M$ exists. Further, by Theorem \ref{c:wpd}, it follows that 
$M$ is equivariantly diffeomorphic to $\bar{M}=\mathcal{Z}_{\bar P^{n}}$ with the product metric, admitting an isometric, effective action of $T^{n+p}$ with free dimension $p$, and  $\bar{N}^{2n}=\bar{M}/T^{p}$ is a torus manifold  of non-negative curvature.
We summarize this result in the following proposition.
\begin{proposition}\label{2a} Let $M^{2n+p}$ be a closed, simply connected,  non-negatively curved manifold admitting an isotropy maximal $T^{n+p}$ action with free rank equal to $p$ and such that the orbit space $P^n$ has $n+p$ facets. 
 Then $M^{2n+p}$ is equivariantly diffeomorphic to the moment angle manifold $\mathcal{Z}_P$. 
\end{proposition}
\noindent The  diagram below illustrates this case.

\vspace{.5cm}

\begin{equation*}
\begin{split}
\xymatrix{
& {} & M^{2n+p}   \ar@/_2pc/[dd]_{/T^{n+p}}  \ar[d]^{/T^{p}}  \ar[r]^{\cong}  &  \bar M^{2n+p}  \ar[r]^{\cong}   \ar[d]^{/T^{p}}  &  \mathcal{Z}_{\bar P^{n}}  = \mathcal{Z}_{P^{n}}    \ar[ld]^{/T^{p}}    \\
& {} & X^{2n}  \ar[d]^{/T^{n}}  & \bar{N}^{2n}    \ar[d]^{/T^{n}}    \\
& {}  &P^{n} \ar[r]^{\cong} &  \bar P^{n}   & {}    &  {}     }
\end{split}
\end{equation*}

\vspace{.5cm}


\subsubsection{\bf Case ($2.b$): The number of facets of ${\bf M/T}$ is strictly greater than the rank of the torus action}
 Let $m$ be the number of facets of $M/T=P$, with $m-n-p>0$.   
It follows from the proof of Theorem \ref{REtorusorbifold}, that $M$ has the rational homotopy type of the base of a principal $T^{m-n-p}$ bundle with total space the corresponding moment angle manifold, which is a product of spheres of dimensions greater than or equal to three.
In particular, using the long exact sequence in homotopy, this tells us that $H^2(M^{2n+p};\zzz)$ is isomorphic to $\zzz^{m-n-p}$.   Note that principal $T^{m-n-p}$-bundles over $M$ are classified by $m-n-p$ elements $\beta_1,\dots, \beta_{m-n-p} \in H^2(M;\zzz)$.  Here each $\beta_i$ can be described as the Euler class of the oriented circle bundle $Y/T^{m-n-p-1} \longrightarrow M$ where $Y$ is the total space and $T^{m-n-p-1} \subset T^{m-n-p}$ is the subtorus with the $i$-th $T^1$-factor deleted.  
It follows that there exists a $T^{m-n-p}$-principal bundle over $M^{2n+p}$ with total space $Y^{n+m}$, a closed, simply connected Riemannian manifold. Further, by Theorem \ref{FukayaYamaguchi}, $Y^{n+m}$ admits a family of almost non-negatively curved metrics and by construction, its quotient space is non-negatively curved. 

As in Case (2a), using Proposition \ref{lstd}, the Equivariant Cross-Sectioning Theorem \ref{t:ecst}, and Theorem \ref{c:wpd}, it follows that $Y^{n+m}$ is equivariantly diffeomorphic to 
$\bar Y^{n+m}$, and that $\bar Y^{n+m}$ admits an isometric and effective  $T^{m}$-action with free dimension equal to $m-n$. Moreover,  $\bar Y^{n+m}$, with the pullback metric from the product metric on 
$\mathcal{Z}_{\bar{P}^n}$, is non-negatively curved. We can now apply Corollary \ref{generalmomentangle} to conclude that $\bar{M}^{2n}$ is the quotient of a free torus action on $\bar{Y}^{n+m}$, which is equivariantly diffeomorphic to the moment angle manifold $\mathcal{Z}_P$. Since $\bar{Y}^{n+m}$ is equivariantly diffeomorphic to $Y^{n+m}$, 
by commutativity of the diagram, we can then conclude that $M$ is equivariantly diffeomorphic to the quotient of the moment angle manifold $\mathcal{Z}_P$ by a free torus action.

 Hence 
$\bar M^{2n} =\bar N^{n+m}/T^{m-n}$ is a non-negatively curved torus manifold. We summarize this result in the following proposition.
\begin{proposition}\label{2b}
Let $M^{2n+p}$ be a closed, simply connected,  non-negatively curved manifold admitting an isotropy maximal $T^{n+p}$ action with free rank equal to $n+p$ and such that the  the orbit space $P^n$ has $m$ facets, where $m > n+p$.
Then there is a smooth principal $T^{m-n-p}$-bundle $\pi: Y^{n+m} \longrightarrow M^{2n+p}$. Further, $Y^{n+m}$ is simply connected and is $T^m$-equivariantly diffeomorphic to the moment angle manifold $\mathcal{Z}_P$.  
Thus, $M^{2n+p}$ is equivariantly diffeomorphic to the quotient of $\mathcal{Z}_P$  by a free $T^{m-n-p}$-action.

\end{proposition}

\noindent The diagram below illustrates this case. 
\vspace{.5cm}

\begin{equation*}
\begin{split}
\xymatrix{
& {} & Y^{n+m}   \ar[d]^{/T^{m-n-p}}  \ar[r]^{\cong}  &  \bar Y^{n+m}    \ar[dd]^{/T^{m-n}}  \ar[rr]^{\cong}  & {}  &  \mathcal{Z}_{\bar P^{n}}  =\mathcal{Z}_{ P^{n}} \ar[lldd]^{/T^{m-n}}   \\
& {} & M^{2n+p}   \ar@/_2pc/[dd]_{/T^{n+p}}  \ar[d]^{/T^{p}}  \\
& {} & X^{2n}  \ar[d]^{/T^{n}}  & \bar N^{2n}    \ar[d]^{/T^{n}}    \\
& {}  &P^{n} \ar[r]^{\cong} &  \bar P^{n}   & {}    &  {}     }
\end{split}
\end{equation*}

\vspace{.5cm}

Propositions \ref{2a} and \ref{2b} give us that $M^{2n+p}$ is equivariantly diffeomorphic to the quotient of the moment angle manifold $\mathcal{Z}_P$ by a freely acting torus, where $P=M/T$. So, in order  to finish the proof of Case (2) of Theorem \ref{t:thma}, it only remains to show that the $T^{m-n-p}$-action on $\mathcal{Z}_P$ is linear. 

We  use the cross-sections $c_1$ and $c_2$ given by the Equivariant Cross Sectioning Theorem \ref{t:ecst} to construct an equivariant diffeomorphism from $M^{2n+p}$ to $\bar{M}^{2n+p}$.  Then $M^{2n+p} \cong \bar{M}^{2n+p} \cong  \mathcal{Z}_{\tilde P^{n}}/T^{m-n-p}$ and 
$T^{m-n-p}$ is a free linear action since it is sub-action of the free linear action of  $T^{m}$ on $\mathcal{Z}_{\tilde P^{n}}$. The  diagram below illustrates the proof.

\begin{equation*}
\begin{split}
\xymatrix{
 & Y^{n+m}    \ar[d]^{/T^{m-n-p}}  \ar[rr]^{\cong}  &{} &  \bar Y^{n+m}  \ar[d]^{/T^{m-n-p}}  \ar@/^2pc/[dd]_{}    \ar[rr]^{\cong} & {} &  \mathcal{Z}_{\bar P^{n}} =\mathcal{Z}_{P^{n}}   \ar[lldd]^{/T^{m-n}}    \\
 & M^{2n+p}   \ar[d]^{/T^{p}}  \ar@{-->}[rr]^{\cong}  & {} &  \bar{M}^{2n+p}  \ar[d]^{} \\
 & X^{2n}  \ar[d]^{/T^{n}} & {} &  \bar N^{2n}    \ar[d]^{/T^{n}}    \\
 &P^{n}  \ar@/^2pc/[uuu]^{c_1}   \ar[rr]^{\cong} & {}    &\bar P^{n} \ar@/^2pc/[uuu]^{c_2}   & {}    &  {}     }
\end{split}
\end{equation*}

\section{Almost isotropy-maximal is isotropy-maximal for $ k\geq\lfloor 2n/3\rfloor$} \label{s5}

 In this section we  establish that for a closed, simply-connected, non-negatively curved $n$-manifold, an almost isotropy-maximal $T^k$-action is isotropy-maximal if the rank of the action is greater than or equal to $\lfloor 2n/3\rfloor$.  This gives us an extension of Theorem \ref{t:thma} to almost isotropy-maximal actions with rank greater than or equal to $\lfloor 2n/3\rfloor$.

\begin{theorem}\label{t:amism}
Let $T^k$ act isometrically, effectively and almost isotropy-maximally on $M^n$, a simply connected, closed, non-negatively curved Riemannian $n$-manifold, with $k\geq\lfloor 2n/3\rfloor$.
Then the action is isotropy-maximal.
\end{theorem}

Before we begin the proof of Theorem \ref{t:amism}, we first need the following topological result for manifolds admitting a disk bundle decomposition. Here we let $\rk(G)$  denote the number of $\zzz$ factors in the  finitely generated abelian group $G$.

\begin{lemma}\label{l:pi1}  Let $M$ be a  manifold with $\rk(H_1(M; \zzz))=k$, $k\in\zzz^+$. If $M$ admits a disk bundle decomposition 
\bdm
M=D(N_1)\cup_{E} D(N_2), 
\edm
where $N_1$, $N_2$ are smooth submanifolds of $M$, and $N_1$ is orientable and of codimension-two, then both 
$\rk(H_1(N_1; \zzz))$ and $\rk(H_1(N_2; \zzz))$ are less than or equal to $k+1$.
 
\end{lemma}
 \begin{proof}
It follows from the Mayer Vietoris sequence of the decomposition and the hypothesis on the rank of $H_1(M)$ that the  sequence 
\beq\label{e:3}
 H_1(E) \rightarrow H_1(N_1)\oplus H_1(N_2)\rightarrow \zzz^k\oplus \Gamma\rightarrow 0
\eeq
 is exact, where $\Gamma$ is a finite abelian group. Since $E$ is a circle bundle over $N_1$, it follows from  the Gysin sequence for homology that  $\rk(H_1(E))\leq 1+\rk(H_1(N_1))$.  The same statement follows for $\rk(H_1(N_2))$ since 
$E$ is either a circle or sphere bundle over $N_2$.  
Using these facts and the exactness of the sequence in Display (\ref{e:3}),  it follows that 
\bdm
\rk(H_1(N_1))+ \rk(H_1(N_2))-k\leq \rk(H_1(E))\leq \rk(H_1(N_1))+ 1,
\edm

\bdm
\rk(H_1(N_1))+ \rk(H_1(N_2))-k\leq \rk(H_1(E))\leq \rk(H_1(N_2))+ 1,
\edm
and the lemma is proven.
\end{proof}

\begin{proof}[Proof of Theorem \ref{t:amism}] The proof is by contradiction. Let $x\in M$ belong to an almost minimal orbit, $T(x)\cong T^k/T^{n-k-1} = T^m$, where $m=2k-n+1$. Then $T_x\cong T^{n-k-1}$ acts on the unit normal 
$S^{2n-2k-2}$ to $T(x)$, and this action is  both isotropy-maximal and of maximal symmetry rank.
 In fact, both the isotropy and the torus action are nested $S^1$-fixed point homogeneous. That is,  there is a nested tower of fixed point sets containing the smallest orbit $T^m$, as follows:
\bdm
T^{m}\subset N^{m+1}\subset N^{m+3}\subset  \cdots \subset N^{n-2}\subset N^n=M^n,
\edm
and $T^m$ and $N^{m+1}$ are both fixed by $T^{k-m}$. 
Note that since $M^n$ is nested $S^1$-fixed point homogeneous, each $N^l$ is a non-negatively curved $S^1$-fixed point homogeneous manifold.  So it follows by Theorem \ref{Spindeler} that each $N^l$ admits a disk bundle decomposition as in the statement of Lemma \ref{l:pi1}, that is,  
$N^l=D(N_1)\cup D(N_2)$, where $N_1=N^{l-2}$.

The induced action of $T^k/T^{k-m}=T^m$ on $N^{m+1}$ is by cohomogeneity one. We claim that this action must have circle isotropy, which in turn implies that the action of $T^k$ must have been isotropy-maximal to begin with.
If it does not, then by the classification of cohomogeneity one torus actions (see  \cite{Mo},  \cite{Neu}, \cite{Pa}, \cite{Pak}), $N^{m+1}$ is equivariantly diffeomorphic to $T^{m+1}$ and has first integer homology group $\zzz^{m+1}$. However,  applying Lemma \ref{l:pi1}  successively to each $N^l$ containing $N^{m+1}$ in the tower for a total of $\frac{n-m-1}{2}=k-m$ times, 
 we see that the number of $\zzz$ factors in $H_1(N^{m+1};\zzz)$ must be less than or equal to $k-m$. Note that $k-m<m+1$ if and only if $k\geq \lfloor 2n/3\rfloor$, giving us the desired contradiction.
\end{proof} 

With this result, we now obtain  the following extension of Theorem \ref{t:thma}. 

\begin{theorem}\label{aimthma} Let $T^k$ act isometrically and effectively on $M^n$, a closed, simply connected, non-negatively curved Riemannian manifold, with $k\geq \lfloor 2n/3\rfloor$. Assume that the action is  almost isotropy-maximal. 
Then $M$ is equivariantly diffeomorphic to a quotient of 
$\mathcal{Z}^m$ by a free linear torus action.
\end{theorem}

The proof of Theorem \ref{aimthma} is completely analogous to the proof of Theorem \ref{t:thma} using Theorem \ref{t:amism} combined with the following remark.
\begin{remark}
By Theorem \ref{t:amism}, 
Corollary \ref{l:torus} holds also when the free rank is equal to $2k-n+1$ for $k\geq \lfloor 2n/3\rfloor$.
\end{remark}


\section{The Proofs of the remaining results} \label{s7}
We now present proofs of Corollary \ref{corb} and Theorem \ref{MSRNN}, as well as a streamlined proof of the Maximal Symmetry Rank Conjecture in dimensions less than or equal to $6$. 
We begin with a proof of Corollary \ref{corb}.
\begin{proof}[Proof of Corollary \ref{corb}]
Note that while we can simply appeal to results in \cite{GGKR} to prove the upper bound on the rank,  
we will give a constructive proof, as it is straightforward and quite simple.

We assume then that $k=\lfloor 2n/3\rfloor+s$, with $s>0$ to obtain a contradiction. Since the action is isotropy-maximal, the free rank is equal to $2k-n$ and 
in particular, we see that $X^{2n-2k}=M^n/T^{2k-n}$ is a torus orbifold. By Theorem \ref{t:thma}, $M^n$ is equivariantly diffeomorphic to the free, linear quotient by a torus of a product of spheres of dimension greater than or equal to  three. This product of spheres can have dimension at most $3n-3k$. So, $n$ must be less than or equal to $3n-3k$. However, a simple calculation shows that for 
$s\geq 1$, $$3n-3k\leq n+2-3s<n,$$ which gives us a contradiction. Hence $k\leq \lfloor 2n/3\rfloor$, as desired.
\end{proof}

Before we prove Theorem \ref{MSRNN}, we first recall Lemma 4.4 from \cite{GGS2}.
\begin{lemma}\cite{GGS2} \label{nonempty} Let $T^{n-3}$ act on $M^{n}$, a closed, simply connected smooth manifold. Then some circle subgroup has non-trivial fixed point set.
\end{lemma}
\noindent We  also need the following proposition, which combines the results of Proposition 4.5 and Corollary 5.6 from \cite{GGS2}.
\begin{proposition}\cite{GGS2}\label{corollary5.6} Let $M^{n}$ be a closed, non-negatively curved manifold with an isometric $T^{n-3}$-action. Suppose that $M/T=S^3$ and that there are isolated $T^{n-4}$ orbits. Then the number of such isolated  $T^{n-4}$ orbits is bounded below by $n-2$ and above by $4$.  In particular, if $n\geq 7$, then there are none.
\end{proposition}

Now, Theorem \ref{MSRNN} follows immediately by combining Theorem \ref{thmc} with the following proposition. 

\begin{proposition}\label{max} Let $M^n$ be a closed, simply connected, non-negatively curved Riemannian manifold  admitting an isometric, effective cohomogeneity three torus action.
If $n\geq 7$, then the action is isotropy-maximal.
\end{proposition}

\begin{proof}

Lemma \ref{nonempty} tells us that a cohomogeneity three torus action on a closed, simply connected manifold must have isotropy subgroups of rank at least $1$.
An effective cohomogeneity $l$ torus action can have isotropy subgroups of rank at most $l$, and an action with isotropy $T^l$ or $T^{l-1}$ will be isotropy-maximal or almost isotropy-maximal, respectively.  Since $n-3\geq\lfloor 2n/3\rfloor$ for $n\geq 7$, Theorem \ref{t:amism} then gives us that an almost  isotropy-maximal action is isotropy-maximal. So, in order to prove the proposition, it suffices to show that such a cohomogeneity three torus action  must be almost isotropy-maximal, that is, we must show that there is an orbit with isotropy subgroup of  rank at least $2$.
We suppose that all orbits have isotropy of rank $1$ to obtain a contradiction. There are two cases to consider: Case (1), where the action has non-isolated orbits of circle isotropy, and Case (2), where the action has only isolated orbits of circle isotropy.

We begin with Case (1), where the action has non-isolated orbits of circle isotropy.
Note that for a cohomogeneity three torus action, if there is only circle isotropy and the corresponding orbit is not isolated, then it  follows that there is a circle acting fixed point homogeneously. 
The corresponding codimension-two fixed point set of the circle, $N^{n-2}$, admits an induced $T^{n-3}/T^1=T^{n-4}$-action of cohomogeneity two.  This action is either free or almost free. By Theorems 12.3 and 12.15 of Conner and Raymond \cite{CR} (see also \cite{OR3}), for the action $(T^{n-1}, N^{n+1})$, with $n\geq 4$, one can find a splitting $T^{n-1}=T^{n-2}\times T^1$ and a finite abelian subgroup $\Gamma\subset T^{n-2}$ so that  $(T^{n-2}, N^{n+1})$ is fibered equivariantly over  $(T^{n-2}, T^{n-2}/\Gamma)$, with fiber $M^3$ and structure group $\Gamma$. Further, since $N^{n+1}$ is non-negatively curved, it follows that $\pi_1(M^3)$ is finite (see \cite{OR1}).
Using the low degree terms in the Leray-Serre-Atiyah-Hirzebruch sequence (see \cite{DK}) we obtain that $\pi_*: H_1(N^{n+1}) \longrightarrow H_1(T^{n-2})$ is surjective, 
where $\pi: N^{n+1} \longrightarrow T^{n-2}$ is the projection map.  Hence $\rk(H_1(N^{n+1})) \ge n-2 > 1$, yielding a contradiction by Lemma \ref{l:pi1}.

We now consider Case (2),  where the action has only isolated orbits of circle isotropy. Recall from Corollary IV.4.7  of \cite{Br} that the quotient space, $M^*$, of a cohomogeneity three $G$-action on a compact, simply connected manifold with connected orbits
is a simply connected $3$-manifold with or without boundary. 
Note that when there is only isolated circle isotropy for a cohomogeneity three torus action, the quotient space will not have boundary and thus, by the resolution of the Poincar\'e conjecture (see Perelman \cite{P1, P2, P3}), we have that $M^*=S^3$.
 Therefore, we may apply Proposition \ref{corollary5.6} to obtain a contradiction.

 The desired result then follows.
\end{proof}

Finally, we present a significantly streamlined proof of the Maximal Symmetry Rank Conjecture for dimensions less than or equal to $6$ (cf. \cite{GGS1}, \cite{GGK}).

 \begin{theorem}\label{msr6} Let $M^n$, $2\leq n \leq 6$ be a closed, simply connected, non-negatively curved manifold admitting an effective, isometric torus action. Then the Maximal Symmetry Rank Conjecture holds.
\end{theorem}

\begin{proof} 
We begin with the following observation.

 \begin{observe}\label{observe} For a cohomogeneity one torus action on a closed, simply connected manifold of dimension $n\geq 2$, without curvature restrictions, the action is isotropy-maximal. This result  also holds for a cohomogeneity two torus action on a closed, simply connected $n$-manifold, for $n\geq 4$ (see Theorem 1.3 in \cite{KMP}).
\end{observe}

Combining  Theorem \ref{thmc}
 with Observation \ref{observe}, we obtain the desired result 
 (cf. \cite{GGS1}, \cite{GGK}).
 \end{proof}
 
  One notes that in order to extend these classification results 
 to higher dimensions, it suffices to show that an action of maximal symmetry rank must be either almost isotropy-maximal or isotropy-maximal.  One possible course of action would be to establish the existence of an 
 upper bound on the free rank of the action that is strictly less than the rank of the action; that is, show that some circle has non-empty fixed point set. However, to date, there are no known obstructions for a cohomogeneity $m$ torus action, $m\geq 4$, on an $n$-dimensional closed, simply connected, non-negatively curved Riemannian manifold with $n\geq 10$ to be free (see Kobayashi and Nomizu \cite{KN}, Angulo-Ardoy, Guijarro and Walschap \cite{AAGW}).



\end{document}